\documentclass[11pt,oneside,reqno]{amsart}
\usepackage{geometry}
\newcommand{\bfG}{\mathbf{G}}

\newcommand{\bfGSp}{\mathbf{GSp}}

\newcommand{\Art}{\mathrm{Art}}
\newcommand{\Gal}{\mathrm{Gal}}
\newcommand{\GNrm}{\mathrm{GNrm}}

\newcommand{\etrm}{\mathrm{\acute{e}t}}

\newcommand{\degrm}{\mathrm{deg}}

\newcommand{\Trrm}{\mathrm{Tr}}

\newcommand{\pr}{\mathrm{pr}}

\newcommand{\ab}{\mathrm{ab}}
\newcommand{\GL}{\mathrm{GL}}

\newcommand{\Res}{\mathrm{Res}}

\newcommand{\ad}{\mathrm{ad}}

\newcommand{\proj}{\mathrm{pr}}

\newcommand{\ssrm}{\mathrm{ss}}
\newcommand{\Aut}{\mathrm{Aut}}
\newcommand{\End}{\mathrm{End}}

\newcommand{\crys}{\mathrm{crys}}

\newcommand{\Nrm}{\mathrm{Nrm}}
\newcommand{\CM}{\mathrm{CM}}

\newcommand{\A}{\mathbb{A}}

\newcommand{\C}{\mathbb{C}}

\newcommand{\F}{\mathbb{F}}

\newcommand{\Q}{\mathbb{Q}}
\newcommand{\R}{\mathbb{R}}

\newcommand{\Z}{\mathbb{Z}}

\newcommand{\GSp}{\mathrm{GSp}}

\newcommand{\Acal}{\mathcal{A}}

\newcommand{\Ccal}{\mathcal{C}}

\newcommand{\Gcal}{\mathcal{G}}
\newcommand{\Hcal}{\mathcal{H}}

\newcommand{\Ncal}{\mathcal{N}}
\newcommand{\Ocal}{\mathcal{O}}

\newcommand{\Scal}{\mathcal{S}}
\newcommand{\Tcal}{\mathcal{T}}

\newcommand{\Gbf}{\mathbf{G}}

\newcommand{\Pbf}{\mathbf{P}}

\newcommand{\Rbf}{\mathbf{R}}

\newcommand{\Ubf}{\mathbf{U}}

\newcommand{\Xbf}{\mathbf{X}}

\newcommand{\hfrak}{\mathfrak{h}}

\newcommand{\Cfrak}{\mathfrak{C}}
\newcommand{\Dfrak}{\mathfrak{D}}
\newcommand{\Gfrak}{\mathfrak{G}}
\newcommand{\Hfrak}{\mathfrak{H}}

\newcommand{\Nfrak}{\mathfrak{N}}

\newcommand{\bcS}{\mathcal{S}}

\newcommand{\bfX}{\mathbf{X}}

\usepackage{imakeidx}
\makeindex[columns=3]
\makeindex[columns=2,options= -s index_style_xyz.ist]
\usepackage[utf8]{inputenc}
\usepackage{tgtermes}
\usepackage[english]{babel}
\usepackage{amsmath,
	mathtools,
	amssymb,
	graphicx,
	bbm,
	xcolor,
	amsthm,
	mathrsfs,
	imakeidx,
	mathdots}
\usepackage[all]{xy}
\usepackage[backref=page]{hyperref}
\hypersetup{}
\usepackage{hyperref}
\usepackage{tikz-cd}
\usetikzlibrary{decorations.pathmorphing}
\usepackage{dynkin-diagrams}
\usepackage{listings}
\usepackage{bm}
\usepackage[normalem]{ulem}
\usepackage{yhmath}
\makeatletter
\@namedef{subjclassname@2020}{%
	\textup{2020} Mathematics Subject Classification}
\makeatother

\theoremstyle{plain}

\newtheorem{theorem}{Theorem}[section]
\newtheorem{corollary}[theorem]{Corollary}
\newtheorem{definition}[theorem]{Definition}
\newtheorem{lemma}[theorem]{Lemma}
\newtheorem{proposition}[theorem]{Proposition}

\newtheorem{remark}[theorem]{Remark}

\newtheorem*{theorem-no-num}{Theorem}
\newtheorem*{proposition-no-num}{Proposition}
\newtheorem*{corollary-no-num}{Corollary}

\newtheorem{taggedtheoremx}{}
\newenvironment{taggedtheorem}[1]
{\renewcommand\thetaggedtheoremx{#1}\taggedtheoremx}
{\endtaggedtheoremx}

\usepackage[shortlabels]{enumitem}

\begin{document}
	\title{Strengthenings of Mazur's Conjecture for Higher Heegner Points}
	\author{Xiaoyu Zhang}
	\address{Universität Duisburg-Essen,
		Fakultät für Mathematik,
		Mathematikcarrée
		Thea-Leymann-Straße 9,
		45127 Essen,
		Germany}
	\email{xiaoyu.zhang@uni-due.de}
	
	\begin{abstract}
    We establish quantitative strengthenings of Mazur's conjecture regarding the non-torsion property of higher Heegner points on modular and Shimura curves, confirming both a vertical version for sufficiently large powers $n$ and a horizontal version for primes $p \gg 1$. Distinct from previous strategies by Cornut and Vatsal that relied on Ratner's theorems on unipotent flows and required restrictive hypotheses on level structures, our approach circumvents these constraints by exploiting the interplay between Galois orbits and Hecke orbits. By quantifying the relative size of these orbits, we reduce the problem to ergodic results concerning the equidistribution of Hecke operators and the classification of joining measures. This method allows for the analysis of simultaneous supersingular reductions without requiring the surjectivity of reduction maps.
	\end{abstract}
	
	\maketitle

	\tableofcontents

	\section{Introduction}	
	Heegner points on modular and Shimura curves encode deep arithmetic information, playing a pivotal role in Iwasawa theory, the Birch--Swinnerton-Dyer conjecture, and the Gross--Zagier formula.

	Let $L$ be an imaginary quadratic field of discriminant $D_{L/\Q}$, and let $N$ be a positive integer such that every prime factor of $N$ splits in $L$. There exists an ideal $\Ncal\subset\Ocal_L$ satisfying $\Ocal_L/\Ncal\simeq\Z/N\Z$. The modular curve $X_0(N)$ parameterizes isogenies of elliptic curves $E_1\to E_2$ with a cyclic kernel of order $N$.
	
	For each positive integer $n$, let $\Ocal_n:=\Z+n\Ocal_L$ denote the order of conductor $n$ in $L$. Set $\Ncal_n:=\Ocal_n\cap\Ncal$ and define the Heegner point on $X_0(N)$ of conductor $n$ by
	\[
	x[n]:=
	[\C/\Ocal_n\to\C/\Ncal_n^{-1}]\in X_0(N).
	\]
	It is well known that $x[n]$ is defined over the ring class field $L[n]$ of $L$ of conductor $n$. Let $J_0(N)$ be the Jacobian of $X_0(N)$. For any abelian variety $A/\Q$ arising as a quotient of $J_0(N)$, we denote the induced morphism by
	\[
	\pi\colon X_0(N)\to J_0(N)\to A.
	\]
	For each $n$, we set $y[n]:=\pi(x[n])$.
	
	Fix a prime number $p$ and let $L[p^\infty]=\bigcup_nL[p^n]$. Let $H[p^\infty]$ denote the maximal anticyclotomic $\Z_p$-extension of $L$; note that $L[p^\infty]/H[p^\infty]$ is a finite extension.  Mazur conjectured (\cite{Mazur1983}) a non-torsion property for these higher Heegner points, a result subsequently established by Cornut (\cite{Cornut2002}):
	\begin{theorem}
		There exists an integer $n\ge0$ such that
		\[
		\Trrm_{L[p^\infty]/H[p^\infty]}(y[p^n])\notin
		A(H[p^\infty])_\mathrm{tors}.
		\]
	\end{theorem}
	
	While Cornut's theorem proves the existence of non-torsion points, it leaves open the quantitative question of their distribution: for how many integers $n$ is the trace $\Trrm_{L[p^\infty]/H[p^\infty]}(y[p^n])$ non-torsion? In particular, are there infinitely many such $n$? In this article, we prove the following strengthening of Mazur's conjecture under mild hypotheses on $p$:
	
	\begin{theorem}[Corollary \ref{vertical Mazur's conjecture, stronger version}]\label{Theorem 1.2}
		Suppose that $p$ is coprime to $ND_{L/\Q}$.
		For $n\gg1$, we have
		\[
		\Trrm_{L[p^\infty]/H[p^\infty]}(y[p^n])\notin A(H[p^\infty])_\mathrm{tors}.
		\]
	\end{theorem}
	We refer the reader to Corollary \ref{vertical Mazur's conjecture, stronger version} for the precise statement. While Mazur's conjecture naturally extends to Shimura curves, Theorem \ref{Theorem 1.2} was previously established in that setting by Cornut and Vatsal \cite[Theorem 4.10]{CornutVatsal2007} only under additional hypotheses on the level structure (assumption (H2) in \emph{loc.cit.}). Their proof (\cite{CornutVatsal2005,CornutVatsal2007}) relies on a delicate analysis of the torus action on adelic points of units in the quaternion algebra. Our approach, by contrast, is distinct and circumvents these restrictions.
	
	We also formulate and prove a \emph{horizontal} version of Mazur's conjecture. Let $L[\infty]=\bigcup_k L[k]$ and fix a set of representatives $\Tcal\subset \Gal(L[\infty]/L)$ for the finite quotient $\Gal(L[1]/L)$ of $\Gal(L[\infty]/L)$.
	
	\begin{theorem}[Corollary \ref{horizontal Mazur's conjecture}, Horizontal Mazur's conjecture]\label{Theorem 1.3}
		For any \emph{prime} $p\gg1$, we have
		\[
		\sum_{\sigma\in\Tcal}\sigma(y[p])\notin A(L[\infty])_\mathrm{tors}.
		\]
	\end{theorem}
	
	Both theorems are consequences of a more general result. Let $n=\prod_{i=1}^s p_i^{e_i}$ be a positive integer coprime to $ND_{L/\Q}$. Define the density term
	\[
	d(n):=\prod_{i=1}^s d(p_i^{e_i}),
	\qquad
	d(p^e)=
	\begin{cases}
		1, & \text{if $p$ is inert in $L$},\\
		\frac{p-1}{p+1}, & \text{if $p$ splits in $L$}.
	\end{cases}
	\]
	Fix a prime $\ell\gg 1$ inert in $L$\footnote{The condition that $\ell$ is inert in $L$ ensures that all Heegner points $x[n]$ have \emph{supersingular} reduction at $\ell$.} (satisfying auxiliary conditions detailed in §\ref{Mazur's conjecture-1}) and a finite subset $\Tcal\subset \Gal(L[\infty]/L)$.
	
	\begin{theorem}[Corollary \ref{generalized Mazur's conjecture}]
		There exists a constant $c(\pi,\ell,\Tcal)\in(0,1)$, depending on the morphism $\pi$, the prime $\ell$, and the set $\Tcal$, such that for any $n\gg1$ coprime to $\ell ND_{L/\Q}$ satisfying $1-d(n)<c(\pi,\ell,\Tcal)$, we have
		\[
		\sum_{\sigma\in\Tcal}\sigma(y[n])\notin A(L[\infty])_\mathrm{tors}.
		\]
	\end{theorem}
	The constant $c(\pi,\ell,\Tcal)$ is given explicitly in \eqref{c(pi,ell,T)}. The condition $1-d(n)<c(\pi,\ell,\Tcal)$ is automatically satisfied if $n$ has no prime factor that splits in $L$; thus, Theorem \ref{Theorem 1.3} and Corollary \ref{horizontal Mazur's conjecture} follow directly.
	\begin{remark}
		Corollaries \ref{generalized Mazur's conjecture}, \ref{vertical Mazur's conjecture}, and \ref{horizontal Mazur's conjecture} are established for arbitrary Shimura curves over $\Q$, not merely modular curves.
	\end{remark}
	
	The proof of Theorem \ref{generalized Mazur's conjecture} reduces to a distribution problem concerning the Galois orbits of the Heegner points $x[n]$. Specifically, let $X_0(N)^\ssrm(\overline{\F}_\ell)$ denote the supersingular locus of $X_0(N)(\overline{\F}_\ell)$, and let $\Gfrak_n \subset \Gal(L[\infty]/L)$ be a specific subset for each $n>0$. In Theorem \ref{Galois orbit equidistributes}, we characterize the following simultaneous supersingular reduction map for $n\gg 1$:
	\[
	\Rbf_{\Tcal}\colon
	\{\nu(x[n])|\nu\in\Gfrak_n\}\to
	\prod_{\sigma\in\Tcal}X_0(N)^\ssrm(\overline{\F}_\ell),
	\quad
	x\mapsto(\sigma(x)^{(\ell)})_{\sigma\in\Tcal}.
	\]
	Here $\sigma(x)^{(\ell)}$ denotes the reduction modulo $\ell$ of $\sigma(x)$. Under a suitable commensurability condition on $\Tcal$, we show that $\Rbf_{\Tcal}$ is surjective for $n\gg 1$.
	
	The strategy of the proof for Theorem \ref{Galois orbit equidistributes} relies on the interplay between the Hecke orbit $T_n(x[1])$ and the Galois orbit $\Gfrak_n(x[n])$. We demonstrate in Theorem \ref{theorem on Hecke orbit vs Galois orbit} that $\Gfrak_n(x[n])$ is contained in $T_n(x[1])$. Furthermore, their relative sizes are quantified by $d(n)$: as $1-d(n)$ decreases, the Galois orbit occupies a larger proportion of the Hecke orbit. Consequently, assuming $1-d(n)<c(\pi,\ell,\Tcal)$, it suffices to prove that the Hecke orbits $(T_n(x[1]))_n$ equidistribute on a specific subset of $\prod_{\sigma\in\Tcal} X_0(N)^\ssrm(\overline{\F}_\ell)$.
	
	To achieve this, we combine the equidistribution results of Clozel--Oh--Ullmo \cite{ClozelOhUllmo2001} with the measure classification of Einsiedler--Lindenstrauss \cite{EinsiedlerLindenstrauss2019}. Heuristically, the former establishes the equidistribution of $T_n(x[1]^{(\ell)})$ on a \emph{single} copy of $X_0(N)^\ssrm(\overline{\F}_\ell)$, while the latter implies that the simultaneous supersingular reductions (twisted by $\Tcal$) generate an \emph{algebraic} measure (\cite[Definition 1.2]{EinsiedlerLindenstrauss2019}) as $n\to+\infty$. We then employ classification results for subgroups of self-products of simple non-commutative groups to conclude.
	
	By refining these arguments, we prove the following stronger version of Corollary \ref{generalized Mazur's conjecture}, which implies Theorem \ref{Theorem 1.2} and Corollary \ref{vertical Mazur's conjecture, stronger version}:
	
	\begin{theorem}[Theorem \ref{S^circ}]
		Fix an infinite set $\Scal^\circ$ of primes $\ell$ such that $x[1]$ has supersingular reduction at every $\ell \in \Scal^\circ$. Then, for any $0<c\ll1$, there exists $n_0$ such that for any $n>n_0$ coprime to the primes in $\Scal^\circ$ and $\Nfrak D_{L/\Q}$ with $1-d(n)<c$, we have
		\[
		\sum_{\sigma\in\Tcal}\sigma(y[n])\notin
		A(L[\infty])_\mathrm{tors}.
		\]
	\end{theorem}
	Notably, the proof of this theorem does not require the surjectivity of $\Rbf_{\Tcal}$ on $\Gfrak_n(x[n])$; rather, it exploits the flexibility of choosing arbitrarily large primes $\ell$ from the infinite set $\Scal^\circ$ (see §\ref{Mazur's conjecture-2}).
	
	This work reduces number-theoretic questions to ergodic theoretic and distributional problems, a theme present in \cite{Cornut2002,Vatsal2002,CornutVatsal2005,CornutVatsal2007}. However, whereas previous works utilized Ratner's theorems on unipotent flows, we employ the classification of joining measures (improving upon \cite{EinsiedlerLindenstrauss2019}) and Hecke operator equidistribution (\cite{ClozelOhUllmo2001}). Our key innovation is the exploitation of the close relationship between Hecke and Galois orbits to deduce distributional properties of the latter. We note that \cite{EinsiedlerLindenstrauss2019} has also been applied to simultaneous supersingular reductions of elliptic curves in \cite{AkaLuethiMichelWieser2022}.
	
	The original conjecture of Mazur has significant applications to the Iwasawa Main Conjecture (\cite{Bertolini1995,NekovarSchappacher1999,Howard2004b}). Potential applications of our results include establishing the non-triviality of Kolyvagin systems for Galois modules of $\Gal(\overline{L[p]}/L[p])$ for $p\gg1$. Furthermore, our techniques may yield non-vanishing results for Rankin--Selberg $L$-values of automorphic representations twisted by characters of sufficiently large conductors, analogous to \cite[Theorem 1.13]{CornutVatsal2007} but in a broader context. We intend to explore these avenues in future work.
	
	The ergodic argument from \cite{Cornut2002} was simplified and generalized to Shimura varieties of Hodge type in \cite{Zhang2025}, and the results of \cite[§2]{CornutVatsal2005} were generalized from $\mathrm{SL}_2$ to reductive groups in \cite[Appendix]{Zhang2025b}. As demonstrated in §\ref{Hecke orbit and Galois orbit}, the relation between Hecke and Galois orbits generalizes naturally to higher rank groups.
	
	\subsection*{Outline}
	The paper is organized as follows. From §\ref{Commensurability criterion} to §\ref{Shimura varieties of Hodge type}, we work in a general setting. In §\ref{Commensurability criterion}, we classify subgroups of a self-product of a simple non-commutative group and extend this to products of simple groups. Section \ref{Hecke operators} introduces Hecke operators. In §\ref{joining of measures}, we establish distribution results for Hecke orbits on \emph{multiple} copies of the double quotient spaces $\Gamma\backslash G(\A_f^\Scal)/U^{\Scal'}$. Section \ref{Shimura varieties of Hodge type} applies these results to Shimura varieties of Hodge type, obtaining explicit equidistribution statements. In §\ref{Heegner points on Shimura curves}, we define Heegner points on Shimura curves. Section \ref{Hecke orbit and Galois orbit} relates Hecke orbits to Galois orbits. In §\ref{Mazur's conjecture-1}, we apply the preceding results to prove variants of Mazur's conjecture. Finally, §\ref{Mazur's conjecture-2} refines the arguments of §\ref{Mazur's conjecture-1} to prove the stronger results.

	\subsection*{Notations}
	For a finite set of places $\bcS$ of $\Q$, we write
	\[
	\Q_\bcS
	=
	\prod_{\ell\in\bcS}
	\Q_p,
	\quad
	\mathbb{A}^\bcS
	=
	\prod_{v\notin\bcS}\,'\Q_v.
	\]
	For an abstract/algebraic group $G$, we write $G^1$ for the derived abstract/algebraic subgroup of $G$ and $Z_G$ for the center of $G$.

	\section{Commensurability criterion}\label{Commensurability criterion}
	Recall that two subgroups $G_1,G_2$ of a group $G$ are \textit{commensurable} if $G_1\cap G_2$ has finite index in both $G_1$ and $G_2$. The \emph{commensurator} of $G_1$ in $G$ is the set of elements $g\in G$ such that $G_1$ and $gG_1g^{-1}$ are commensurable; we denote it by $\Ccal(G_1)$. If $G_2=gG_1g^{-1}$ for some $g\in G$, then $G_1$ and $G_2$ are commensurable if and only if $g\in\Ccal(G_1)$.

	For a group $G$ and a finite set $\mathcal{T}$, write
	\[
	\Delta=\Delta^\Tcal
	\colon G\rightarrow\prod_{\sigma\in\mathcal{T}} G
	\]
	for the diagonal map. For any $\sigma_0\in\mathcal{T}$, write
	\[
	\mathrm{pr}_{\sigma_0}\colon\prod_{\sigma\in\mathcal{T}} G\rightarrow G
	\]
	for the projection onto the $\sigma_0$-th component.
	For a non-empty subset $\mathcal{T}'\subset\mathcal{T}$, define a subgroup of $\prod_{\sigma\in\mathcal{T}} G$ by
	\[
	G^{\mathcal{T}'}
	=
	\left\{
	(g_{\sigma})
	\in
	\prod_{\sigma\in\mathcal{T}} G
	\mid
	g_{\sigma}=1\text{ for all }\sigma\notin\mathcal{T}'
	\right\}.
	\]
	We write $\Delta^{\mathcal{T}'}\colon G\to G^{\mathcal{T}'}\hookrightarrow G^{\mathcal{T}}$ for the partial diagonal map.

	A subgroup $H$ of $G^{\mathcal{T}}$ is called a \textit{product of diagonals} if there exist pairwise disjoint, non-empty subsets $\mathcal{T}_1,\cdots,\mathcal{T}_r$ of $\mathcal{T}$ such that
	$H=\prod_{i=1}^r\Delta^{\mathcal{T}_i}(G)$.
	(We do not require the union of the $\mathcal{T}_i$ to be all of $\mathcal{T}$.)
	Then we have
	\begin{proposition}\label{product of diagonals}
		Suppose that $G$ is simple and non-commutative. Then a subgroup $H$ of $G^{\mathcal{T}}$ is normalized by $\Delta^{\mathcal{T}}(G)$ if and only if it is a product of diagonals.
	\end{proposition}
	\begin{proof}
		This is \cite[Proposition 3.10]{Cornut2002}.
	\end{proof}

	In the next two subsections, §\ref{G simple} and §\ref{G not simple}, we treat separately the case where $G$ is simple and the case where $G$ is a product of simple groups; the latter is built on the former.

	\subsection{$G$ simple}\label{G simple}
	From now on, we assume that $G$ is a simple, non-commutative $p$-adic Lie group that is generated by one-parameter adjoint unipotent subgroups
	(in the sense of \cite{Ratner1995}). Let $(\Gamma_{\sigma})_{\sigma\in\mathcal{T}}$ be a finite family of discrete, cocompact subgroups of $G$, and set
	\[
	\Gamma=\prod_{\sigma\in\mathcal{T}}\Gamma_{\sigma}.
	\]
	Then $\Gamma\backslash G^{\mathcal{T}}$ is compact.
	We then define a partition of the finite set $\mathcal{T}$ by
	\begin{equation}\label{partition of T}
		\mathcal{T}=\bigsqcup_{i=1}^r\mathcal{T}_i,
	\end{equation}
	such that $\sigma,\sigma'\in\mathcal{T}_i$ if and only if $\Gamma_{\sigma}$ and $\Gamma_{\sigma'}$ are commensurable (this is well-defined since commensurability is an equivalence relation).
	\begin{lemma}\label{closure of diagonal image depending on commensurability}
		Suppose $\Tcal=\{\sigma,\sigma'\}$ and write $\Gamma_{\sigma,\sigma'}=\Gamma_{\sigma}\times\Gamma_{\sigma'}$. The closure of $\Gamma_{\sigma,\sigma'}\Delta(G)$ in $G^2=G^\Tcal$ is\footnote{Throughout this article, the closure is taken with respect to the natural topology (either $p$-adic or adelic).}
		\[
		\begin{cases*}
			G^2,
			&
			if $\Gamma_{\sigma}$ and $\Gamma_{\sigma'}$ are not commensurable;
			\\
			\Gamma_{\sigma,\sigma'}\Delta(G),
			&
			if $\Gamma_{\sigma}$ and $\Gamma_{\sigma'}$ are commensurable.
		\end{cases*}
		\]
	\end{lemma}
	\begin{proof}
		By assumption, $G$ is generated by one-parameter adjoint unipotent subgroups, thus we can apply Ratner's theorem on unipotent flows (\cite[Theorem 2]{Ratner1995}) to get that the closure of $\Gamma_{\sigma,\sigma'}\Delta(G)$ in $G^2$ is of the form $\Gamma_{\sigma,\sigma'}H$ for some closed subgroup $H$ of $G^2$. Clearly $H$ is normalized by $\Delta^{\mathcal{T}}(G)$ and thus by Proposition \ref{product of diagonals}, $H$ is either $\Delta(G)$ or $G^2$ (there are only two partitions of the set $\{\sigma,\sigma'\}$).

		Observe that the following natural bijection is a homeomorphism
		\[
		\left(
		\Gamma_{\sigma}\bigcap\Gamma_{\sigma'}
		\right)
		\backslash G
		\rightarrow
		\left(
		\Delta(G)\bigcap\Gamma_{\sigma,\sigma'}
		\right)
		\backslash\Delta(G).
		\]
		On the other hand, since $\Gamma_{\sigma,\sigma'}$ is discrete in $G^2$, $\Delta(G)$ is open in $\Gamma_{\sigma,\sigma'}\Delta(G)$ and the following bijection is again a homeomorphism (because it is an open continuous map)
		\[
		\left(
		\Delta(G)\bigcap\Gamma_{\sigma,\sigma'}
		\right)
		\backslash\Delta(G)
		\rightarrow
		\Gamma_{\sigma,\sigma'}\backslash\Gamma_{\sigma,\sigma'}\Delta(G).
		\]
		In particular, $(\Gamma_{\sigma}\bigcap\Gamma_{\sigma'})\backslash G$ is compact if and only if
		$\Gamma_{\sigma,\sigma'}\backslash\Gamma_{\sigma,\sigma'}\Delta(G)$ is compact, if and only if
		$\Gamma_{\sigma,\sigma'}\backslash\Gamma_{\sigma,\sigma'}\Delta(G)$ is closed in 
		$\Gamma_{\sigma,\sigma'}\backslash G^2$.

		Now if $\Gamma_{\sigma}$ and $\Gamma_{\sigma'}$ are commensurable, then $(\Gamma_{\sigma}\bigcap\Gamma_{\sigma'})\backslash G$ is compact, so $\Gamma_{\sigma,\sigma'}\backslash\Gamma_{\sigma,\sigma'}\Delta(G)$ is closed in 
		$\Gamma_{\sigma,\sigma'}\backslash G^2$, thus $H=\Delta(G)$.
		If they are not commensurable, then $(\Gamma_{\sigma}\bigcap\Gamma_{\sigma'})\backslash G$ is not compact, so $\Gamma_{\sigma,\sigma'}\backslash\Gamma_{\sigma,\sigma'}\Delta(G)$ is not closed in $G^2$, therefore we must have $H=G^2$.
	\end{proof}

	\begin{proposition}\label{closure of diagonal map}
		The closure of $\Gamma\Delta^{\mathcal{T}}(G)$ in $G^{\mathcal{T}}$ has the form $\Gamma H$, where $H$ is a product of diagonals $\prod_{i=1}^r\Delta^{\mathcal{T}_i}(G)$ with $\mathcal{T}_i$ as above. In particular, if each $\mathcal{T}_i$ contains only one element, then $\Gamma\Delta^{\mathcal{T}}(G)$ is dense in $G^{\mathcal{T}}$.
	\end{proposition}
	\begin{proof}
		As in the previous lemma, the closure has the form $\Gamma H$ for some closed subgroup $H$ of $G^{\mathcal{T}}$ of the form
		$H=\prod_{i=1}^{r'}\Delta^{\mathcal{T}_i'}(G)$
		for some \textit{partition} of $\mathcal{T}$:
		\[
		\mathcal{T}=\bigsqcup_{i=1}^{r'}\mathcal{T}_i'.
		\]

		Now take any two elements $\sigma,\sigma'\in\mathcal{T}$ and consider the projection onto the components of $\sigma$ and $\sigma'$:
		\[
		\mathrm{pr}_{\sigma}
		\times
		\mathrm{pr}_{\sigma'}
		\colon
		G^{\mathcal{T}}
		\rightarrow
		G^2.
		\]
		Clearly, the image of $\Gamma H$ under this projection map is equal to the closure of $\Gamma_{\sigma,\sigma'}\Delta(G)$.
		It follows from Lemma \ref{closure of diagonal image depending on commensurability} that $\sigma,\sigma'\in\mathcal{T}_i$ if and only if the closure of $\Gamma_{\sigma,\sigma'}\Delta(G)$ is \textit{not} equal to $G^2$, if and only if $\Gamma_{\sigma}$ and $\Gamma_{\sigma'}$ are commensurable, if and only if $\sigma,\sigma'\in\mathcal{T}_{i'}$ for some $i'=1,\cdots,r$.
		Therefore, each $\mathcal{T}_i'$ is contained in some $\mathcal{T}_{i'}$, and hence the two partitions $\mathcal{T}=\bigsqcup_{i=1}^r\mathcal{T}_i$ and $\mathcal{T}=\bigsqcup_{i=1}^{r'}\mathcal{T}_i'$ coincide. This finishes the proof.
	\end{proof}

	\subsection{$G$ a product of simple groups}\label{G not simple}
	For later use, we also need to treat the case where $G$ is not simple, but rather a product of simple groups.
	In this subsection we write
	\[
	G=\prod_{j=1}^t G^{(j)},
	\]
	where each $G^{(j)}$ is a simple, non-commutative $p$-adic Lie group generated by one-parameter adjoint unipotent subgroups, as in the previous subsection (in the sense of \cite{Ratner1995}).

	\begin{proposition}\label{H is a product of diagonals}
		A subgroup $H$ of $G^\Tcal$ is normalized by $\Delta^\Tcal(G)$ if and only if it is of the form
		\[
		H=\prod_{j=1}^t\prod_{i=1}^{r_j}\Delta^{\Tcal_i^{(j)}}(G^{(j)}),
		\]
		where, for each $j=1,\cdots,t$, the sets $\Tcal_1^{(j)},\cdots,\Tcal_{r_j}^{(j)}$ are pairwise disjoint, non-empty subsets of $\Tcal$.
	\end{proposition}
	\begin{proof}
		For $j=1,\cdots,t$, set
		\[
		H^{(j)}=H\cap (G^{(j)})^\Tcal,
		\]
		a subgroup of $(G^{(j)})^\Tcal$.
		Since $H$ is normalized by $\Delta^\Tcal(G)$, the subgroup $H^{(j)}$ is normalized by $\Delta^\Tcal(G^{(j)})$.
		Thus, by Proposition \ref{product of diagonals}, we have that $H^{(j)}$ is a product of diagonals in $(G^{(j)})^\Tcal$:
		\[
		H^{(j)}=\prod_{i=1}^{r_j}\Delta^{\Tcal_i^{(j)}}(G^{(j)}).
		\]
		Here $\Tcal_1^{(j)},\cdots,\Tcal_{r_j}^{(j)}$ are pairwise disjoint, non-empty subsets of $\Tcal$.
		
		Set $\widetilde{H}=\prod_{j=1}^t H^{(j)}$. For any $h\in\widetilde{H}$, write $h=h_1\cdots h_t$ with $h_j\in (G^{(j)})^\Tcal$.
		Fix $j\in\{1,\cdots,t\}$. Then, for any $g_j\in\Delta^\Tcal(G^{(j)})$, we have
		\[
		g_jhg_j^{-1}
		=
		h_1\cdots h_{j-1}(g_jh_jg_j^{-1})h_{j+1}\cdots h_t\in H.
		\]
		Thus $g_jhg_j^{-1}h^{-1}=g_jh_jg_j^{-1}h_j^{-1}\in H\cap (G^{(j)})^\Tcal=H^{(j)}$ for any $g_j\in\Delta^\Tcal(G^{(j)})$.
		We claim that $h_j\in H^{(j)}$. This implies $H=\widetilde{H}$ and completes the proof.
		
		To prove the claim, write $h_j=(h_{j,\sigma})_{\sigma\in\Tcal}$ with $h_{j,\sigma}\in G^{(j)}$, and consider two cases:
		\begin{enumerate}
			\item 
			Fix $\sigma,\sigma'\in\Tcal_i^{(j)}$. Then for any $g\in G^{(j)}$, we have
			\[
			gh_{j,\sigma}g^{-1}h_{j,\sigma}^{-1}
			=
			gh_{j,\sigma'}g^{-1}h_{j,\sigma'}^{-1}.
			\]
			In particular, $h_{j,\sigma}^{-1}h_{j,\sigma'}$ commutes with every $g\in G^{(j)}$. Since $G^{(j)}$ is simple, we must have
			\[
			h_{j,\sigma}=h_{j,\sigma'}.
			\]

			\item 
			Fix $\sigma\notin\bigsqcup_{i=1}^{r_j}\Tcal_i^{(j)}$. Then for any $g\in G^{(j)}$, we have
			\[
			gh_{j,\sigma}g^{-1}h_{j,\sigma}^{-1}=1,
			\]
			and hence
			\[
			h_{j,\sigma}=1.
			\]
		\end{enumerate}
		From these two cases, we deduce that $h_j=(h_{j,\sigma})_{\sigma\in\Tcal}$ lies in $H^{(j)}=\prod_{i=1}^{r_j}\Delta^{\Tcal_i^{(j)}}(G^{(j)})$.
	\end{proof}

	As in Lemma \ref{closure of diagonal image depending on commensurability}, we have the following.
	\begin{lemma}
		Suppose $\Tcal=\{\sigma,\sigma'\}$. Then $\Gamma_{\sigma,\sigma'}\Delta(G)$ is closed in $G^2$ if and only if $\Gamma_{\sigma}$ and $\Gamma_{\sigma'}$ are commensurable.
	\end{lemma}
	\begin{proof}
		The proof is exactly the same as in Lemma \ref{closure of diagonal image depending on commensurability}, except that we do not know whether the closure of $\Gamma_{\sigma,\sigma'}\Delta(G)$ is equal to $G^2$ when $\Gamma_{\sigma}$ and $\Gamma_{\sigma'}$ are \emph{not} commensurable.
	\end{proof}

	\begin{corollary}\label{Gamma_sigma and Gamma_sigma' comm or not, not simple}
		Suppose $\Tcal=\{\sigma,\sigma'\}$ and that for any $j=1,\cdots,t$, the subgroup $\Gamma G^{(j)}$ is dense in $G$. Then the closure of $\Gamma_{\sigma,\sigma'}\Delta(G)$ in $G^2$ is
		\[
		\begin{cases*}
			\Gamma_{\sigma,\sigma'}\Delta(G),
			&
			if $\Gamma_{\sigma}$ and $\Gamma_{\sigma'}$ are commensurable;
			\\
			G^2,
			&
			if $\Gamma_{\sigma}$ and $\Gamma_{\sigma'}$ are not commensurable.
		\end{cases*}
		\]
	\end{corollary}
	\begin{proof}
		It suffices to treat the case where $\Gamma_{\sigma}$ and $\Gamma_{\sigma'}$ are not commensurable.
		Since $\Gamma G^{(j)}$ is dense in $G$, so are $\Gamma_{\sigma}G^{(j)}$ and $\Gamma_{\sigma'}G^{(j)}$.
		By the proof of Lemma \ref{closure of diagonal image depending on commensurability}, the closure of $\Gamma_{\sigma,\sigma'}\Delta(G)$ in $G^2$ is of the form $\Gamma_{\sigma,\sigma'}H$ for some closed subgroup $H\subset G^2$ containing $\Delta(G)$ and normalized by $\Delta(G)$.
		Thus, by Proposition \ref{H is a product of diagonals}, $H$ is of the form
		\[
		H=\prod_{j=1}^t\prod_{i=1}^{r_j}\Delta^{\Tcal_i^{(j)}}(G^{(j)}).
		\]
		Since $H\supset\Delta(G)$, we have $\Tcal=\bigsqcup_{i=1}^{r_j}\Tcal_i^{(j)}$ for each $j$.
		Since $\Gamma_{\sigma}$ and $\Gamma_{\sigma'}$ are not commensurable, we have $H\neq\Delta(G)$. Hence there exists some $j_0\in\{1,\cdots,t\}$ such that $r_{j_0}=2$. In particular, $H$ contains $(G^{(j_0)})^{\Tcal}$.
		However, by assumption, $\Gamma_{\sigma,\sigma'}(G^{(j_0)})^{\Tcal}$ is dense in $G^2$, and therefore $H=G^2$. This proves the corollary.
	\end{proof}

	Then, as in Proposition \ref{closure of diagonal map}, we have the following.
	\begin{proposition}\label{closure of diagonal map, not simple}
		Suppose that for any $j=1,\cdots,t$, the subgroup $\Gamma G^{(j)}$ is dense in $G$.
		Then the closure of $\Gamma\Delta^\Tcal(G)$ in $G^\Tcal$ has the form $\Gamma H$, where $H$ is a closed subgroup of $G$ given by
		\[
		H=\prod_{i=1}^r\Delta^{\Tcal_i}(G)
		\]
		for the partition $\Tcal=\bigsqcup_{i=1}^r\Tcal_i$ as in \eqref{partition of T}.
		In particular, if each $\Tcal_i$ contains only one element, then $\Gamma\Delta^{\Tcal}(G)$ is dense in $G^\Tcal$.
	\end{proposition}
	\begin{proof}
		The proof is the same as that of Proposition \ref{closure of diagonal map}, taking Corollary \ref{Gamma_sigma and Gamma_sigma' comm or not, not simple} into account.
	\end{proof}

	\section{Hecke operators}\label{Hecke operators}
	Let $\bfG$ be an algebraic group over $\Q$, and let $\bcS\subset\bcS'$ be finite sets of places of $\Q$. Let $U^{\bcS'}$ be a compact open subgroup of $\bfG(\mathbb{A}^{\bcS'})$, and let $\Gamma\subset\bfG(\Q)$ be a subgroup of the form $\Gamma=\Gbf(\Q)\cap V$ for some compact open subgroup $V\subset \bfG(\Q_{\bcS})$.

	For any $a\in\bfG(\mathbb{A}^{\bcS})$, the \emph{Hecke correspondence} associated with $a$ is the diagram
	\[
	\begin{tikzcd}
		&
		\Gamma
		\backslash
		\bfG(\mathbb{A}^{\bcS})/(U^{\bcS'}\bigcap aU^{\bcS'} a^{-1})
		\arrow[ld,"\pi_1"']
		\arrow[rd,"\pi_2\colon g\mapsto ga"]
		&
		\\
		\Gamma
		\backslash
		\bfG(\mathbb{A}^{\bcS})/U^{\bcS'}
		&
		&
		\Gamma
		\backslash
		\bfG(\mathbb{A}^{\bcS})/U^{\bcS'}
	\end{tikzcd}
	\]
	Here the left arrow is the natural projection map, and the right arrow is the composition of the natural projection with right-multiplication by $a$.
	
	\begin{definition}
		Let $a\in\bfG(\mathbb{A}^{\bcS})$. The \emph{Hecke operator $T_a$} attached to $a$ is defined as follows: for any $x\in\Gamma\backslash\bfG(\mathbb{A}^{\bcS})/U^{\bcS'}$,
		\[
		T_a(x)
		=
		\pi_2\bigl(\pi_1^{-1}(x)\bigr)
		\subset
		\Gamma\backslash\bfG(\mathbb{A}^{\bcS})/U^{\bcS'}.
		\]
		The degree $\deg(a)=\deg(T_a)$ of $T_a$ is
		\[
		\deg(a)
		=
		[U^{\bcS'}\colon U^{\bcS'}\cap aU^{\bcS'} a^{-1}].
		\]
	\end{definition}

	The Hecke operator also induces an operator on functions $f\in L^2\bigl(\Gamma\backslash\bfG(\mathbb{A}^{\bcS})/U^{\bcS'}\bigr)$ by
	\[
	T_a(f)(x)
	:=
	\frac{1}{\deg(a)}
	\sum_{y\in T_a(x)}
	f(y).
	\]
	This in turn gives rise to a measure $\mu_{a,x}$ on $\Gamma\backslash\bfG(\mathbb{A}^{\bcS})/U^{\bcS'}$, defined by the requirement that, for any such $f$,
	\[
	\int_{\Gamma\backslash\bfG(\mathbb{A}^{\bcS})/U^{\bcS'}}
	f\, d\mu_{a,x}
	:=	
	T_a(f)(x).
	\]

	\section{Joining of measures}\label{joining of measures}
	Let $\bfG_0,\bfG_1,\cdots,\bfG_r$ be connected, almost simple, and noncommutative algebraic groups over $\Q$ such that $\bfG_i(\R)$ is \emph{compact} and $\Gbf_i^1$ is simply connected for all $i=0,\cdots,r$. Assume that there is a finite set $\bcS$ of places of $\Q$ such that, for any \emph{finite} place $v\notin\bcS$, the groups $\bfG_i(\Q_v)$ are all mutually isomorphic ($i=0,1,\cdots,r$). In what follows, we fix isomorphisms
	\begin{equation}\label{G_i(A_f^S)=G_0(A_f^S)}
		\Gbf_0(\A^\Scal)\simeq\Gbf_i(\A^\Scal),
		\quad
		\forall
		i=1,\cdots,r.
	\end{equation}
	We then fix a finite set $\bcS'$ of places of $\Q$ containing $\bcS$ such that $\bfG_i(\Q_{\bcS'})$ is non-compact for all $i=0,1,\cdots,r$. In particular, each $\bfG_i^1$ satisfies the strong approximation property with respect to $\bcS'$.

	We fix a compact open subgroup $U_0^{\bcS'}\subset\bfG_0(\mathbb{A}_f^{\bcS'})$ and denote by $U_i^{\bcS'}$ its image under the isomorphism $\bfG_0(\mathbb{A}_f^{\bcS'})\simeq\bfG_i(\mathbb{A}_f^{\bcS'})$. We set $\bfG=\prod_{i=1}^r\bfG_i$ and $\Gbf^1=\prod_{i=1}^r\Gbf_i^1$, and we write $U_0^{\Scal',1}=U_0^{\Scal'}\bigcap\Gbf_0^1(\A_f^{\Scal'})$ (similarly for $U_i^{\Scal',1}$). We put
	\begin{align*}
		\Gamma_i
		&=\bfG_i(\Q)\bigcap U_i^{\bcS'},
		\quad
		\Gamma_i^1=\Gamma_i\bigcap\Gbf_i^1(\Q),
		\quad
		\Gamma=\prod_{i=1}^r\Gamma_i,
		\quad
		\Gamma^1=\Gamma\bigcap\Gbf^1(\Q);
		\\
		G_i
		&=\bfG_i^1(\Q_{{\bcS'}\backslash\bcS}),
		\quad
		G=\prod_{i=1}^rG_i;
		\\
		\bfX_i
		&=
		\bfG_i^1(\Q)
		\backslash\bfG_i^1(\mathbb{A}_f^{\bcS})/U_i^{\bcS',1}
		=
		\Gamma_i^1
		\backslash
		G_i,
		\quad
		\bfX=\prod_{i=1}^r\bfX_i.
	\end{align*}
	We assume in the following that
	\begin{equation}\label{G vs G^1 for X}
		\Xbf_i
		=
		Z_{\Gbf_i}(\Q_{\Scal'\backslash\Scal})\Gamma_i
		\backslash
		\Gbf_i(\Q_{\Scal'\backslash\Scal})
		=
		Z_{\Gbf_i}(\A_f^\Scal)\Gbf_i(\Q)
		\backslash
		\Gbf_i(\A_f^\Scal)/U^{\Scal'},
		\quad
		\forall
		i=0,\cdots,r.
	\end{equation}
	Here the second identity holds provided that we assume
	\[
	\Gbf_i(\A_f^{\Scal'})
	=
	Z_{\Gbf_i}(\A_f^{\Scal'})\Gbf_i(\Q)
	\Gbf_i^1(\A_f^{\Scal'})U^{\Scal'},
	\quad
	\forall
	i=0,\cdots,r.
	\]
	We write $\mu_{\bfX_i}$ for the measure on $\bfX_i$ induced by the Haar measure on $G_i$. This is the unique measure that is invariant under the action of $G_i$. Similarly, we write
	\[
	\mu_{\bfX}=\mu_{\bfX_1}\times\cdots\times\mu_{\bfX_r},
	\]
	the unique measure invariant under the action of $G$.
	We have the following result. map::
	\[
	\Pi
	\colon
	\bfG_0(\mathbb{A}_f^{\bcS})/U_0^{\bcS'}
	\to
	\prod_{i=1}^r
	\bfG_i(\mathbb{A}_f^{\bcS})/U_i^{\bcS'}
	\to
	\bfX.
	\]
	The first map is the diagonal map induced by the isomorphisms in (\ref{G_i(A_f^S)=G_0(A_f^S)}) while the second map is the natural projection taking into account (\ref{G vs G^1 for X}).

	For any $x\in\bfG_0(\mathbb{A}_f^{\bcS})/U_0^{\bcS'}$ and any $a\in\bfG_0(\mathbb{A}_f^{\bcS})$, we have an induced measure $\mu_{\Pi(T_a(x))}$ on $\bfX$ defined by
	\[
	\int_{\bfX}f(z)d\mu_{\Pi(T_a(x))}(z)
	:=
	\frac{1}{\deg(a)}
	\sum_{y\in T_a(x)}f(\Pi(y)),
	\quad
	\forall
	f\in L^2(\bfX).
	\]

	Note that $G_i$ acts on $\bfX_i$ by right multiplication, and similarly $G$ acts on $\bfX$. In particular, $\bfG_0(\Q_{\bcS'})$ acts on each $\bfX_i$ and acts diagonally on $\bfX$.

	We have the following equidistribution result:
	\begin{theorem}\label{Equidisitrbution of Hecke points by Clozel-Oh-Ullmo}
		Suppose that $r=1$ and that (\ref{G vs G^1 for X}) holds. Then, for any sequence $(a_n)_n$ in $\bfG_0(\mathbb{A}_f^{\bcS})$ such that $\deg(a_n)\to +\infty$, the measures $\mu_{\Pi(T_{a_n}(x))}$ converge to $\mu_{\bfX}$.
	\end{theorem}
	\begin{proof}
		We need to show that for any $f\in L^2(\bfX)$ with compact support,
		\begin{equation}\label{mu_{a_n,x} goes to mu_{bfX_i}}
			\int_{\bfX}f(z)d\mu_{\Pi(T_{a_n}(x))}(z)
			\rightarrow
			\int_{\bfX}f(z)d\mu_{\bfX}(z).
		\end{equation}
		We have a natural projection
		\[
		Z_{\Gbf}(\Q_{{\bcS'}\backslash\bcS}\times\R)
		\Gamma\backslash(\Gbf(\Q_{{\bcS'}\backslash\bcS})\times\bfG(\R))
		\longrightarrow
		\bfX,
		\]
		so we may view $f$ as an element of $L^2\bigl(Z_{\Gbf}(\Q_{{\bcS'}\backslash\bcS}\times\R)\Gamma\backslash(G\times\bfG(\R))\bigr)$, using the assumption that each $\bfG_i(\R)$ is compact. Then (\ref{mu_{a_n,x} goes to mu_{bfX_i}}) follows from \cite[Theorem 1.7 and §4.7]{ClozelOhUllmo2001}.\footnote{In \cite{ClozelOhUllmo2001}, the authors use $\Gbf(\Q_{\Scal'\backslash\Scal})$ instead of $\Gbf(\A^\Scal)$. It is straightforward to translate between the two using the strong approximation property for $\Gbf$. Moreover, \cite[§4.7]{ClozelOhUllmo2001} only treats the case $\Gbf_0=\GL_n$ or $\GSp_{2n}$; the same argument applies to our $\Gbf_0$, assuming (\ref{G vs G^1 for X}).}
	\end{proof}

	Since $\bfG_0$ is almost simple, there is a simple subgroup $\bfG_0'$ of $\bfG_0$ such that $\bfG_0'\subset\bfG_0\subset\mathrm{Aut}(\bfG_0')$. Similarly, we have simple subgroups $\bfG_i'$ of $\bfG_i$. Under the identifications (\ref{G_i(A_f^S)=G_0(A_f^S)}), we have isomorphisms
	\[
	\Gbf_0'(\A_f^\Scal)\simeq\Gbf_i'(\A_f^\Scal),
	\quad
	\forall
	i=1,\cdots,r.
	\]
	We set $\Gamma_i':=\Gamma_i\bigcap\Gbf_i'(\Q)$ and $\bfG'=\prod_{i=1}^r\bfG_i'$.
	Then we can and will identify $\Gbf'(\A_f^\Scal)$ with $\Gbf_0(\A_f^\Scal)^r$. We have a partition
	\begin{equation}\label{partition of {1,2,...,r}}
		\{1,\cdots,r\}=\Tcal_1\bigsqcup\cdots\bigsqcup\Tcal_k,
	\end{equation}
	such that $i,j\in\Tcal_s$ if and only if $\Gamma_i'$ and $\Gamma_j'$ are commensurable (viewed as subgroups in $\Gbf_0'(\A_f^\Scal)$). We define the following subgroup of $\Gbf'(\A_f^\Scal)$
	\[
	\Gbf'_\Gamma(\A_f^\Scal)
	:=
	\prod_{i=1}^k\Delta^{\Tcal_i}(\Gbf_0'(\A_f^\Scal)).
	\]
	Similarly, one defines $\Gbf'_\Gamma(\Q_v)$ for any place $v$ of $\Q$ not in $\Scal$.

	By \cite[Theorem 1.4]{EinsiedlerLindenstrauss2019}, we have
	\begin{theorem}\label{limit measure invariant under G'}
		Assume the following:
		\begin{enumerate}
			\item 
			$\Gbf_i=\Gbf_i'$ for all $i=0,\cdots,r$;

			\item 
			there are two finite places $\ell_1,\ell_2\in{\bcS'}\backslash\bcS$ such that $\bfG_0^1(\Q_{\ell_1})$ and $\bfG_0^1(\Q_{\ell_2})$ are both non-compact;

			\item 
			for $\ell'=\ell_1,\ell_2$, the quotient group $\Gbf_0^1(\Q_{\ell'})/Z_{\Gbf_0^1(\Q_{\ell'})}$ is simple.
		\end{enumerate}
		Fix an element $x\in\bfG_0(\mathbb{A}_f^{\bcS})/U_0^{\bcS'}$ and a sequence of elements $(a_n)_n$ in $\bfG_0(\mathbb{A}_f^{\bcS})$ ($n=1,2,\cdots$) such that $\deg(a_n)\to +\infty$. Then any limit measure of the sequence $\mu_{\Pi(T_{a_n}(x))}$ is the normalized Haar measure on a single orbit $Z_{\Gbf}(\Q_{{\bcS'}\backslash\bcS})\Gamma H$ for some closed subgroup $H\subset\Gbf^1(\Q_{\Scal_1'})$ such that
		\[
		HZ_{\Gbf(\Q_{\Scal_1'})}
		=\Gbf_\Gamma^1(\Q_{\Scal_1'})Z_{\Gbf(\Q_{\Scal_1'})}.
		\]
		Here $\Scal_1'=\{\ell_1,\ell_2\}$ and $\Gbf_\Gamma^1(\Q_{\Scal_1'})=\Gbf_\Gamma(\Q_{\Scal_1'})\bigcap\Gbf^1(\Q_{\Scal_1'})$.
	\end{theorem}
	\begin{proof}
		Let $\mu$ be a limit measure of the sequence $\mu_{\Pi(T_{a_n}(x))}$; say $\mu_{\Pi(T_{a_{n_k}}(x))}\to\mu$ for a subsequence $n_k\to\infty$. By Theorem \ref{Equidisitrbution of Hecke points by Clozel-Oh-Ullmo}, $\mu$ is a joining of the measures $\mu_{\bfX_1},\cdots,\mu_{\bfX_r}$.
		Since each $\mu_{\bfX_i}$ is $\bfG_0^1(\Q_{\bcS'})$-ergodic by Theorem \ref{Equidisitrbution of Hecke points by Clozel-Oh-Ullmo}, the measure $\mu$ is $\Delta(\Gbf_0^1(\Q_{\Scal'}))$-ergodic. Thus $\mu$ is invariant under the diagonal action of $\bfG_0^1(\Q_{\bcS'})$; in other words, for any $f\in L^2(\bfX)$, we have
		\[
		\int_{\bfX}f(zg)d\mu(z)=\int_{\bfX}f(z)d\mu(z),
		\]
		for any $g\in\Delta(\bfG_0^1(\Q_{\bcS'}))$.

		For each $i=1,\cdots,r$, fix an arbitrary compact open subgroup $U_{i,{\bcS'}\backslash\bcS'_1}$ of $\bfG_i(\Q_{{\bcS'}\backslash\bcS'_1})$. Write then $U_{{\bcS'}\backslash\bcS'_1}=\prod_{i=1}^rU_{i,{\bcS'}\backslash\bcS'_1}$ and $\overline{\mu}$ for the quotient measure of $\mu$ on $\bfX/U_{{\bcS'}\backslash\bcS'_1}$. Thus, to prove the theorem, it suffices to show that $\overline{\mu}$ is invariant under $\prod_{i=1}^r\bfG_i'(\Q_{\bcS'_1})$ for arbitrary $U_{i,{\bcS'}\backslash\bcS'_1}$.

		For two elements $\alpha_1\in\bfG_0(\Q_{\ell_1})$ and $\alpha_2\in\bfG_0(\Q_{\ell_2})$, we have the following proper homomorphisms ($i=1,\cdots,r$)
		\[
		\psi_i
		\colon
		\mathbb{Z}^2
		\to
		\bfG_i(\Q_{\ell_1})\times\bfG_i(\Q_{\ell_2}),
		\quad
		(m_1,m_2)
		\mapsto
		(\alpha_1^{m_1},\alpha_2^{m_2}).
		\]
		We can take $\alpha_1,\alpha_2$ such that $\psi=(\psi_1,\cdots,\psi_r)$ is of class-$\mathcal{A}'$ (with respect to the set $\bcS'_1=\{\ell_1,\ell_2\}$) in the sense of \cite[Definition 1.3]{EinsiedlerLindenstrauss2019}.
		Clearly $\overline{\mu}$ is invariant and ergodic under $\mathbb{Z}^2$ by the preceding claim. Therefore by \cite[Theorem 1.4]{EinsiedlerLindenstrauss2019}, $\overline{\mu}$ is an algebraic measure defined over $\Q$, in particular, there is a closed subgroup $H$ of $\bfG^1(\Q_{\bcS'_1})$ such that the support of $\overline{\mu}$ is of the form $\Gamma\backslash\Gamma Hg$ for some $g\in\bfG(\Q_{\bcS'_1})$. We next show
		\[
		HZ_{\Gbf(\Q_{\Scal_1'})}
		=\bfG_\Gamma^1(\Q_{\bcS'_1})Z_{\Gbf(\Q_{\Scal_1'})}.
		\]

		For $\ell\in\Scal_1'$, we write $H_\ell$ for the projection of $H$ to the $\ell$-th component. We write in the following
		\[
		\Pbf H
		:=H/(H\bigcap Z_{\Gbf(\Q_{\Scal_1'})}),
		\quad
		\Pbf\Gbf_\Gamma^1(\Q_{\Scal_1'})
		:=\Gbf_\Gamma^1(\Q_{\Scal_1'})/(\Gbf_\Gamma^1(\Q_{\Scal_1'})\bigcap Z_{\Gbf(\Q_{\Scal_1'})}).
		\]
		Similarly, one has $\Pbf H_\ell$ and $\Pbf\Gbf_\Gamma^1(\Q_\ell)$ for $\ell\in\Scal_1'$. So we need to show 
		\[
		\Pbf H=\Pbf\Gbf_\Gamma^1(\Q_{\Scal_1'}).
		\]

		By construction, $H$ contains the diagonal image $\Delta(\Gbf_0^1(\Q_{\Scal_1'}))$ of $\bfG_0^1(\Q_{\bcS'_1})$, thus it is invariant under the conjugation of $\Delta(\bfG_0(\Q_{\bcS'_1}))$. It follows that $\Pbf H$ is invariant under the conjugation of $\Delta(\Gbf_0(\Q_{\Scal_1'}))$. By Proposition \ref{product of diagonals} and our assumption (3), we know that for any $\ell\in\Scal_1'$, $\Pbf H_\ell$ is a product of diagonals, in other words, there is partition
		$\{1,2,\cdots,r\}=\bigsqcup_{i}^{k(\ell)}\mathcal{T}_i(\ell)$ such that
		\[
		H_\ell Z_{\Gbf(\Q_\ell)}
		=\prod_{i=1}^{k(\ell)}
		\Delta^{\mathcal{T}_i(\ell)}(\bfG_0(\Q_\ell))
		Z_{\Gbf(\Q_\ell)}.
		\]
		By assumption (3), we have
		\[
		HZ_{\Gbf(\Q_{\Scal_1'})}
		=
		\prod_{i=1}^{k'}\Delta^{\Tcal_i'}(\Gbf_0(\Q_{\Scal_1'}))
		Z_{\Gbf(\Q_{\Scal_1'})},
		\]
		where $\{1,\cdots,r\}=\bigsqcup_{i=1}^{k'}\Tcal_i'$ is the partition obtained by intersecting the two partitions $\{\Tcal_i(\ell_1)\}$ and $\{\Tcal_i(\ell_2)\}$. 
		By Proposition \ref{closure of diagonal map, not simple}, we have the containment
		\[
		HZ_{\Gbf(\Q_{\Scal_1'})}
		\supset
		\prod_{i=1}^{k}\Delta^{\Tcal_i}(\Gbf_0(\Q_{\Scal_1'}))
		Z_{\Gbf(\Q_{\Scal_1'})}
		=
		\Gbf_\Gamma(\Q_{\Scal'\backslash\Scal})
		Z_{\Gbf(\Q_{\Scal'\backslash\Scal})}.
		\]
		We claim that this partition $\{\Tcal_i'\}$ is the same as the partition $\{\Tcal_i\}$ in (\ref{partition of {1,2,...,r}}). For this, it suffices to show that for any $j_1\neq j_2=1,\cdots,r$, if $j_1,j_2\in\Tcal_i$, then $j_1,j_2\in\Tcal_{i'}'$ for some $i'$. By definition, $\Gamma_{j_1}$ and $\Gamma_{j_2}$ are commensurable. We write
		\[
		\Gamma_{j_1,j_2}:=\Gamma_{j_1}\bigcap\Gamma_{j_2}
		\]
		and define $\Xbf'$ to be $\Xbf$ with the components $\Xbf_{j_1}$ and $\Xbf_{j_2}$ replaced by $\Xbf_{j_1}':=Z_{\Gbf_{j_1}}(\Q_{{\bcS'}\backslash\bcS})
		\Gamma_{j_1,j_2}\backslash\Gbf_{j_1}(\Q_{\Scal'\backslash\Scal})$
		and
		$\Xbf_{j_2}':=Z_{\Gbf_{j_2}}(\Q_{{\bcS'}\backslash\bcS})
		\Gamma_{j_1,j_2}\backslash\Gbf_{j_2}(\Q_{\Scal'\backslash\Scal})$ respectively. Similar to the map $\Pi$, we have
		\[
		\Pi'
		\colon
		\Gbf_0(\A_f^\Scal)/U_0^{\Scal'}
		\to
		\Xbf'.
		\]
		Then for any limit measure $\mu'$ of the sequence $(\mu_{\Pi'(T_{a_{n_k}}(x))})_{k>0}$ on $\Xbf'$, we write $\mu'_{j_1,j_2}$ for its projection to $\Xbf_{j_1}'\times\Xbf_{j_2}'\simeq(\Xbf_{j_1}')^2$. Since the composition of $\Pi'$ with the projection of $\Xbf'$ to $\Xbf_{j_1}'\times\Xbf_{j_2}'$ is the diagonal map, by Proposition \ref{Equidisitrbution of Hecke points by Clozel-Oh-Ullmo}, the measure $\mu'_{j_1,j_2}$ is the normalized measure on $\Gamma_{j_1,j_2}^2\Delta_{j_1,j_2}(\Gbf_0^1(\Q_{\Scal'\backslash\Scal}))g'$ for some $g'\in\Gbf_{j_1}^1(\Q_{\Scal'\backslash\Scal})
		\times
		\Gbf_{j_2}^1(\Q_{\Scal'\backslash\Scal})$. Here
		\[
		\Delta_{j_1,j_2}
		\colon
		\Gbf_0^1(\Q_{\Scal'\backslash\Scal})
		\to
		\Gbf_{j_1}^1(\Q_{\Scal'\backslash\Scal})
		\times
		\Gbf_{j_2}^1(\Q_{\Scal'\backslash\Scal})
		\]
		is the diagonal map.

		Along the natural map $\Xbf'\to\Xbf$, the measure $\mu'$ on $\Xbf'$ is mapped to the measure $\mu$ on $\Xbf$. By the above discussion, we know that the projection $\mu_{j_1,j_2}$ of $\mu$ to $\Xbf_{j_1}\times\Xbf_{j_2}$ is the normalized measure on $(\Gamma_{j_1}\times\Gamma_{j_2})\Delta_{j_1,j_2}(\Gbf_0(\Q_{\Scal'\backslash\Scal}))g'$. It follows that the projection of $H$ to the components $(j_1,j_2)$ is contained in $\Delta_{j_1,j_2}(\Gbf_0(\Q_{\Scal'\backslash\Scal}))(Z_{\Gbf_{j_1}(\Q_{\Scal'\backslash\Scal})}\times Z_{\Gbf_{j_2}(\Q_{\Scal'\backslash\Scal})})$. We deduce immediately that $j_1,j_2\in\Tcal_{i'}'$ for some $i'$.

		To finish the proof, we need to show that we can take $g\in Z_{\Gbf(\Q_{\Scal'\backslash\Scal})}$.
		For this, consider any $j_1,j_2\in\Tcal_i$. From the above discussion on $\mu_{j_1,j_2}'$, we know that the two projections of $\Gamma Hg$ to the component $j_1$ and $j_2$ must be equal to each other, so we can take the two components $g_{j_1}$ and $g_{j_2}$ of $g$ at $j_1$ and $j_2$ respectively to be equal to each other. In other words, we can assume $g\in\prod_{i=1}^k\Delta^{\Tcal_i}(\Gbf_0^1(\Q_{\Scal'\backslash\Scal}))$. Thus we have
		\[
		\Gamma HZ_{\Gbf(\Q_{\Scal'\backslash\Scal})}
		=
		\Gamma HgZ_{\Gbf(\Q_{\Scal'\backslash\Scal})},
		\]
		which finishes the proof of the theorem.
	\end{proof}

	\section{Shimura varieties of Hodge type}\label{Shimura varieties of Hodge type}
	We first fix some notation, following closely
	\cite{Kisin2010,Kisin2017,Wortmann2013,Zhang2025}.
	We fix a Shimura datum of Hodge type $(G,X)$ with $G$ a connected reductive group over $\Q$. We also fix a symplectic embedding
	\[
	(G,X)\hookrightarrow(\bfGSp_{L,\Q},S^\pm),
	\]
	where $S^\pm$ is the Siegel upper half space of rank equal to $\mathrm{rk}(L)$. We can and will assume $\mathrm{rk}(L)\ge4$ in the following.
	We fix a compact open subgroup $U$ of $G(\A_f)$ and write $Sh_{U,\C}=Sh_{U,\C}(G,X)$ for the Shimura variety attached to the triple $(G,X,U)$, whose $\C$-points are given by
	\begin{equation}\label{S_K(C) and double coset}
		Sh_{U,\C}(\C)
		=
		G(\Q)\backslash(X\times G(\A_f)/U),
	\end{equation}
	which has a canonical model defined over the reflex field $F=F(G,X)$. We denote this canonical model by $ Sh_{U,F}$. Similarly, for a compact open subgroup $\Ubf$ of $\bfGSp_L(\A_f)$, we have a canonical model $ Sh_{\Ubf,\Q}$ over $\Q$. Suppose that $U$ is contained in $\Ubf$. Then we have a morphism of Shimura varieties
	\[
	Sh_{U,F}\to Sh_{\Ubf,F}= Sh_{\Ubf,\Q}\times_\Q E.
	\]

	For a place $v$ of $F$ over a rational prime $\ell$ such that $G$ is unramified at $\ell$ and $U=U^\ell U_\ell$ with $U_\ell$ hyperspecial at $\ell$, Kisin (\cite{Kisin2010}) constructed an integral canonical model of $Sh_{U,F}$ over $\Ocal_{(v)}$, the localization of $\Ocal_F$ at $v$. We denote this integral model by $Sh_{U,\Ocal_{(v)}}$. We assume henceforth that $\Ubf=\Ubf^\ell\Ubf_\ell$ with $\Ubf_\ell$ also hyperspecial at $\ell$.
	We write $\Acal_{\Ubf,\Ocal_{(v)}}\to Sh_{\Ubf,\Ocal_{(v)}}$ for the universal abelian scheme over $Sh_{\Ubf,\Ocal_{(v)}}$ and we denote by
	\[
	\Acal_{K,\Ocal_{(v)}}\to Sh_{U,\Ocal_{(v)}}
	\]
	its pullback along $Sh_{U,\Ocal_{(v)}}\to Sh_{\Ubf,\Ocal_{(v)}}$. We fix a reductive model $\Gcal$ of $G$ over $\Z_{(\ell)}$ such that $U_\ell=\Gcal(\Z_\ell)$. Then there is a finite set of tensors $s^{(\ell)}=(s_i^{(\ell)})_i\subset L^\otimes_{\Z_{(\ell)}}$ such that $\Gcal$ is identified with the stabilizer in $\GL(L_{\Z_{(\ell)}})$ of these tensors $s^{(\ell)}$ via our symplectic embedding $G\hookrightarrow\bfGSp_{L,\Q}$ (\cite[§2.3.1 and §2.3.2]{Kisin2010}).\footnote{Here $L_{\Z_{(\ell)}}^\otimes$ is the direct sum of all $\Z_{(\ell)}$-modules arising from $L_{\Z_{(\ell)}}$ by taking a finite number of the following operations: duals, tensor products, symmetric powers and exterior powers.}

	For each point $x\in Sh_{U,\Ocal_{(v)}}(\overline{\Q})$, Kisin constructed tensors $s^{(\ell)}_{\ell,x}=(s^{(\ell)}_{i,\ell,x})_i\subset(H_\etrm^1(\Acal_x,\Z_\ell)\otimes_{\Z_\ell}B_\crys)^\otimes$ (\cite[§1.3.6]{Kisin2017}). Here $\Acal_x$ is the specialization of $\Acal_{K,\Ocal_{(v)}}$ to the point $x$. For each prime $\ell'\neq\ell$, Kisin also constructed tensors $s^{(\ell)}_{\ell',x}=(s^{(\ell)}_{i,\ell',x})\subset(H^1_\etrm(\Acal_x,\Q_{\ell'}))^\otimes$ (\emph{loc.cit}).

	Similarly, for each point $x\in Sh_{U,\Ocal_{(v)}}(\overline{\F}_\ell)$, Kisin constructed tensors $s^{(\ell)}_{0,x}=(s^{(\ell)}_{i,0,x})_i\subset(H^1_\crys(\Acal_x/W(\overline{\F}_\ell))\otimes_{W(\overline{\F}_\ell)}B_\crys)^\otimes$ (\cite[§1.3.10]{Kisin2017}). For each prime $\ell'\neq\ell$, there are tensors $s^{(\ell)}_{\ell',x}=(s^{(\ell)}_{i,\ell',x})_i\subset H^1_\etrm(\Acal_x,\Q_{\ell'})^\otimes$ (\cite[§1.3.6]{Kisin2017}). Similarly, we have tensors $s^{(\ell)}_{\ell,x}=(s^{(\ell)}_{i,\ell,x})$ (\emph{loc.cit}). We have a natural isomorphism of $W(\overline{\F}_\ell)$, resp. $\Q_{\ell'}$-modules
	\begin{equation}\label{relation between Hodge tensors}
		L_{W(\overline{\F}_\ell)}\simeq
		H^1_\crys(\Acal_x/W(\overline{\F}_\ell)),\,
		\text{resp. }
		L_{\Q_{\ell'}}\simeq
		H^1_\etrm(\Acal,\Q_{\ell'}),
	\end{equation}
	taking the tensors $s^{(\ell)}_i$ to $s^{(\ell)}_{i,0,x}$, resp. $s^{(\ell)}_{i,\ell',x}$ for all $i$. These tensors for $x\in Sh_{U,\Ocal_{(v)}}(\overline{\F}_\ell)$ are constructed by considering a lifting $\widetilde{x}\in Sh_{U,\Ocal_{(v)}}(\overline{\Q})$ of $x$ and using $\ell'$-adic comparison theorems (\cite[Proposition 1.3.7]{Kisin2017}).

	For $x$ either in $Sh_{U,\Ocal_{(v)}}(\overline{\Q})$ or in $Sh_{U,\Ocal_{(v)}}(\overline{\F}_\ell)$, we define an algebraic subgroup $I_x$ of $\Aut_\Q(\Acal_x)$, consisting of elements fixing the tensors $s^{(\ell)}_{i,0,x}$ and $s^{(\ell)}_{i,\ell',x}$ for any $\ell'\neq\ell$.

	\begin{lemma}\label{I_x(Q_p)=G(Q_p) for supersingular}
		Let $x\in Sh_{U,\Ocal_{(v)}}(\overline{\F}_\ell)$. If $\Acal_x$ is a supersingular abelian variety, then for any prime $\ell'\neq\ell$, we have an isomorphism of algebraic groups over $\Q_{\ell'}$::
		\[
		I_{x/\Q_{\ell'}}:=I_x\times_\Q\Q_{\ell'}\simeq G_{\Q_{\ell'}}.
		\]
	\end{lemma}
	\begin{proof}
		By definition, for any $\Q_{\ell'}$-algebra $R$, $I_{x/\Q_{\ell'}}(R)$ consists of elements in $\Aut_{\Q_{\ell'}}(\Acal_x)(R)$ fixing the tensors $s^{(\ell)}_{i,0,x}$ and $s^{(\ell)}_{i,\ell',x}$. By the assumption that $\mathrm{rk}(L)\ge4$, there exists a supersingular elliptic curve $E(\ell)$ over $\overline{\F}_\ell$ such that $\Acal_x$ is isogenous to $E(\ell)^{\mathrm{rk}(L)/2}$ (by the theorem of Deligne--Ogus--Shioda). In particular, we have an isomorphism of algebraic groups over $\Q$:
		\[
		\Aut_\Q(\Acal_x)
		\simeq
		\GL_{\mathrm{rk}(L)/2}(B(\ell)),
		\]
		where $B(\ell)=\End_{\overline{\F}_\ell}(E(\ell))^\circ$. Since $B(\ell)$ is unramified at $\ell'\neq\ell$, we have an isomorphism $B(\ell)\otimes_\Q\Q_{\ell'}\simeq\mathrm{M}_2(\Q_{\ell'})$ and thus
		\[
		\Aut_{\Q_{\ell'}}(\Acal_x)
		\simeq
		\GL_{\mathrm{rk}(L)/2}(B(\ell)\otimes\Q_{\ell'})
		\simeq
		\GL_{\mathrm{rk}(L)}(\Q_{\ell'}).
		\]

		Since we have the symplectic embedding $G\hookrightarrow\bfGSp_{L,\Q}$, without loss of generality, we can assume that one of the tensors, say, $s^{(\ell)}_{i_0}$, corresponds to the symplectic pairing $\langle-,-\rangle$ on $L_{\Z_{(\ell)}}$. So using the isomorphisms in (\ref{relation between Hodge tensors}), the stabilizer in $\Aut_{\Q_{\ell'}}(\Acal_x)$ of the tensor $s^{(\ell)}_{i_0,\ell',x}$ is exactly $\bfGSp_{L/\Q_{\ell'}}$. It follows that the stabilizer in $\Aut_{\Q_{\ell'}}(\Acal_x)$ of all the tensors $(s^{(\ell)}_{i,\ell',x})_i$ is exactly the group $G_{\Q_{\ell'}}$. Thus we have $I_{x/\Q_{\ell'}}\simeq G_{\Q_{\ell'}}$.
	\end{proof}
	We deduce that $I_x$ is a form of $G$. Moreover, if $\Acal_x$ is supersingular, then $I_x(\R)$ is a closed subgroup of $\Gbf_{\Acal_x}(\R)$, which is compact. It follows that $I_x(\R)$ is also compact.

	For a prime number $\ell$ such that $U=U^\ell U_\ell$ and a point $x\in Sh_{U,\Ocal_{(v)}}(K)$, where $K$ is an algebraically closed field, we write
	\[
	\Hcal^{G,\ell}(\Acal_x)
	\]
	for the set of points $x'\in Sh_{U,\Ocal_{(v)}}(K)$ such that there is a quasi-isogeny $\phi\colon\Acal_x\to\Acal_{x'}$ of degree prime to $\ell$ sending the tensors $s^{(\ell)}_{0,x}$, resp. $s^{(\ell)}_{\ell',x}$, to the tensors $s^{(\ell)}_{0,x'}$, resp. $s^{(\ell)}_{\ell',x'}$ (for any $\ell'$). For $x\in Sh_{U,\Ocal_{(v)}}(K)$ with $K=\overline{\Q}$ or $\overline{\F}_\ell$, by \cite[(1) \& (2)]{Zhang2025} and \cite[Corollary 1.4.2 \& Propositions 2.1.3, 2.1.5]{Kisin2017}, we have a natural bijection
	\[
	\Theta_x\colon
	\Hcal^{G,\ell}(\Acal_x)
	\simeq
	I_x(\Q)\backslash G(\A_f^{(\ell)})U/U.
	\]
	Here we implicitly identify $I_x(\Q_{\ell'})$ with $G(\Q_{\ell'})$ for $\ell'\neq\ell$ using Lemma \ref{I_x(Q_p)=G(Q_p) for supersingular}.
	Note that the right hand side is in natural bijection with $(I_x(\Q)\bigcap U^p)\backslash G(\Q_p)/U_p$.

	For $K=\overline{\Q}$, the point $x\in Sh_{U,\Ocal_{(v)}}(K)\subset Sh_{U,\Ocal_{(v)}}(\C)$ corresponds to an element $(h_x,g_x)\in X\times G(\A_f)$ under the identification (\ref{S_K(C) and double coset}). Then for any $g\in G(\Q_p)$, the point $\Theta_x^{-1}(g)\in Sh_{U,\Ocal_{(v)}}(K)\subset Sh_{U,\Ocal_{(v)}}(\C)$ corresponds to the element $(h_x,g_xg)\in X\times G(\A_f)$ under (\ref{S_K(C) and double coset}). Moreover, for $x\in Sh_{U,\Ocal_{(v)}}(\overline{\F}_\ell)$ and a lifting $\widetilde{x}\in Sh_{U,\Ocal_{(v)}}(\overline{\Q})$ of $x$, we have the following commutative diagram (\cite[Corollary 1.4.12]{Kisin2017})
	\begin{equation}\label{commutative diagram for I_x and I_{widetilde{x}}}
		\begin{tikzcd}
			I_{\widetilde{x}}(\Q)\backslash G(\A_f^{(\ell)})U/U
			\arrow[d]
			&
			\Hcal^{G,\ell}(\Acal_{\widetilde{x}})
			\arrow[l,"\Theta_{\widetilde{x}}"',"\simeq"]
			\arrow[r,hookrightarrow]
			\arrow[d,"\Rbf_\ell"]
			&
			Sh_{U,\Ocal_{(v)}}(\Ocal_{\overline{\Q}})
			\arrow[d,"\Rbf_\ell"]
			\\
			I_x(\Q)\backslash G(\A_f^{(\ell)})U/U
			&
			\Hcal^{G,\ell}(\Acal_x)
			\arrow[l,"\Theta_x"',"\simeq"]
			\arrow[r,hookrightarrow]
			&
			Sh_{U,\Ocal_{(v)}}(\overline{\F}_\ell)
		\end{tikzcd}
	\end{equation}
	The left vertical arrow is the natural projection map induced by the inclusion $I_{\widetilde{x}}(\Q)\hookrightarrow I_x(\Q)$.

	We consider a point $x\in Sh_{U,\Ocal_{(v)}}(\overline{\F}_\ell)$ with $\Acal_x$ supersingular, $U_p=I_x(\Z_p)$. We write $I_x^1$ for the derived subgroup of $I_x$ and $\proj\colon I_x\to I_x/I_x^1=:I_x^\ab$ for the projection map, where $I_x^\ab$ is a torus over $\Q$. We assume
	\begin{equation}\label{condition on similitude of I_x}
		\proj(I_x(\A_f^{(\ell)}))
		=
		\proj(I_x(\widehat{\Z}^{(\ell)}))\proj(I_x(\Q)\bigcap U_\ell).
	\end{equation}
	Then the inclusion $I_x^1\hookrightarrow I_x$ induces a bijection
	\begin{equation}\label{I_x and I_x^1}
		(I_x^1(\Q)\bigcap U_\ell)\backslash I_x^1(\A_f^{(\ell)})/I_x^1(\widehat{\Z}^{(\ell)})
		\simeq
		(I_x(\Q)\bigcap U_\ell)\backslash I_x(\A_f^{(\ell)})/I_x(\widehat{\Z}^{(\ell)}).
	\end{equation}
	Since $I_x^1(\widehat{\Z}^{(\ell)})$ contains the center of $I_x^1(\A_f^{(\ell)})$, we have
	\begin{equation}\label{I_x^1 and PI_x^1}
		(I_x^1(\Q)\bigcap U_\ell)\backslash\Pbf I_x^1(\A_f^{(\ell)})/\Pbf I_x^1(\widehat{\Z}^{(\ell)})
		\simeq
		(I_x^1(\Q)\bigcap U_\ell)\backslash I_x^1(\A_f^{(\ell)})/I_x^1(\widehat{\Z}^{(\ell)}).
	\end{equation}
	Putting all these together, we have a bijection
	\begin{equation}\label{PI_x^1 vs PI_x}
		(I_x^1(\Q)\bigcap U_\ell)\backslash\Pbf I_x^1(\A_f^{(\ell)})/\Pbf I_x^1(\widehat{\Z}^{(\ell)})
		\simeq
		(I_x(\Q)\bigcap U_\ell)\backslash\Pbf I_x(\A_f^{(\ell)})/\Pbf I_x(\widehat{\Z}^{(\ell)}).
	\end{equation}
	We write $I_x^{1,\ad}$ for the adjoint quotient of $I_x^1$. Then $I_x^{1,\ad}$ is the derived subgroup of $I_x^\ad$. By \cite{Humphreys1969}, we know that
	\begin{equation*}
		\Aut(I_x^{\ad}(\Q))=I_x^{\ad}(\Q)/Z_{I_x^{\ad}}(\Q)\rtimes\Aut_\Q(\Dfrak_x^{\ad}),
	\end{equation*}
	where $\Dfrak_x^{\ad}$ is the Dynkin diagram of $I_x^{\ad}$ and $\Aut_\Q(\Dfrak_x^{\ad})$ is the group of automorphisms over $\Q$ of $\Dfrak_x^{1,\ad}$.

	We fix a place $v(\ell)$ of $\overline{\Q}$ over $\ell$ and a point $x_0\in Sh_{U,F}(\overline{\Q})$ such that $x_0$ has \emph{supersingular} good reduction at $v(\ell)$. We also fix a non-empty finite subset $\Tcal=\{\sigma_1,\cdots,\sigma_r\}$ of $Z_{I_{x_0}}(\A_f)$. We identify $\Tcal$ with the set $\{1,2,\cdots,r\}$.
	Moreover, we write $x_0^{(\ell)}$ for the reduction modulo $v(\ell)$ of $x_0$. We write $\Gamma=I_{x_0}(\Q)\bigcap U_\ell$, $\Gamma^1=\Gamma\bigcap I_{x_0}^1(\Q)$ and for $i=1,\cdots,r$,
	\begin{align*}
		\Gamma^{(\ell)}
		&
		:=I_{x_0^{(\ell)}}(\Q)\bigcap U_\ell,
		\quad
		\Gamma_i^{(\ell)}:=\sigma_i^{-1}\Gamma\sigma_i,
		\\
		\Gamma^{(\ell),1}
		&
		:=\Gamma\bigcap I_{x_0^{(\ell)}}^1(\Q),
		\quad
		\Gamma^{(\ell),1}_i:=\sigma_i^{-1}\Gamma^1\sigma_i.
	\end{align*}	
	Then we have a partition
	\[
	\Tcal=\bigsqcup_{i=1}^k\Tcal_{i},
	\]
	such that $\sigma,\sigma'\in\Tcal_{i}$ if and only if
	$\Gamma_\sigma$ and $\Gamma_{\sigma'}$ are commensurable.

	We have the following simultaneous reduction map of the $p$-adic Hecke orbit of $\Acal_{x_0}$:
	\[
	\Rbf_{\Tcal}^G
	\colon
	\Hcal^{G,\ell}(\Acal_{x_0})
	\to
	\prod_{\sigma\in\Tcal}\Hcal^{G,\ell}(\Acal_{(\sigma x_0)^{(\ell)}}),
	\quad
	\Acal_{x_0'}
	\mapsto
	(\Acal_{(\sigma x_0')^{(\ell)}})_{\sigma\in\Tcal}.
	\]
	Moreover, assuming (\ref{condition on similitude of I_x}) and using (\ref{commutative diagram for I_x and I_{widetilde{x}}}), (\ref{I_x and I_x^1}), (\ref{I_x^1 and PI_x^1}), and (\ref{PI_x^1 vs PI_x}), we have the following commutative diagram:
	\begin{equation}\label{commutative diagram for R^G_T}
		\begin{tikzcd}
			\Hcal^{G,\ell}(\Acal_{x_0})
			\arrow[r,"\Rbf^G_{\Tcal}"]
			\arrow[d,"\simeq"]
			&
			\prod_{\sigma\in\Tcal}
			\Hcal^{G,\ell}(\Acal_{(\sigma x_0)^{(\ell)}})
			\arrow[d,"\simeq"]
			\\
			\Pbf\Gamma
			\backslash \Pbf G(\A_f^{(\ell)})/\Pbf U^{(\ell)}
			\arrow[r,"\pr_\Tcal"]
			\arrow[d,"\simeq"]
			&
			\prod_{i=1}^r
			\Pbf\Gamma_i^{(\ell)}
			\backslash \Pbf G(\A_f^{(\ell)})/\Pbf U^{(\ell)}
			\arrow[d,"\simeq"]
			\\
			\Pbf\Gamma^1
			\backslash \Pbf G^1(\A_f^{(\ell)})/\Pbf U^{(\ell),1}
			\arrow[r,"\pr_\Tcal^1"]
			&
			\prod_{i=1}^r
			\Pbf\Gamma^{(\ell),1}_i
			\backslash \Pbf G^1(\A_f^{(\ell)})/\Pbf U^{(\ell),1}
		\end{tikzcd}
	\end{equation}
	Here $U^{(\ell),1}=U^{(\ell)}\bigcap G^1(\A_f^{(\ell)})$, and the horizontal maps $\pr_\Tcal$ and $\pr_\Tcal^1$ are the simultaneous projection maps.
	For each $\Tcal_i$, we define a map
	\[
	\widetilde{\Delta}^{\Tcal_i}
	\colon
	\Hcal^{G,\ell}(\Acal_{x_0^{(\ell)}})
	\to
	\prod_{\sigma\in\Tcal_i}
	\Hcal^{G,\ell}(\Acal_{(\sigma x_0)^{(\ell)}}).
	\]
	This corresponds, under the commutative diagram (\ref{commutative diagram for R^G_T}), to the following \emph{twisted} diagonal map:
	\[
	\Delta^{\Tcal_i}\colon
	\Pbf\Gamma\backslash \Pbf G(\A_f^{(\ell)})/\Pbf U^{(\ell)}
	\to
	\prod_{\sigma\in\Tcal_i}
	\Pbf\Gamma_\sigma\backslash \Pbf G(\A_f^{(\ell)})/\Pbf U^{(\ell)},
	\quad
	\overline{g}\mapsto
	(\overline{\sigma^{-1}g})_{\sigma\in\Tcal_i}.
	\]

	By Theorem \ref{limit measure invariant under G'}, we have
	\begin{theorem}\label{T_{a_n(x_0)} equidistribute on Shimura variety}
		Assume the following:
		\begin{enumerate}
			\item
			(\ref{condition on similitude of I_x}) holds;
			
			\item
			$G^1$ is simple and simply connected;
			
			\item
			there exist two primes $\ell_1,\ell_2$ different from $\ell$ such that $G(\Q_{\ell_1})$ and $G(\Q_{\ell_2})$ are both non-compact and, modulo the center, they are both simple and non-commutative;
			
			\item
			set $\Scal=\{\ell\}$, $\Scal'=\{\ell,\ell_1,\ell_2\}$, and $\Gbf_0=\cdots=\Gbf_r=\Pbf I_{x_0}$; then (\ref{G vs G^1 for X}) holds.
		\end{enumerate}
		For any sequence $(a_n)_n$ in $G(\A_f^{(\ell)})$ such that $\deg(a_n)\to+\infty$, the sequence $(T_{a_n}(x_0))_n$ equidistributes on the following subset of $\prod_{\sigma\in\Tcal}\Hcal^{G,\ell}(\Acal_{x_0^{(\ell)}})$:
		\[
		\prod_{i=1}^k\widetilde{\Delta}^{\Tcal_i}(\Hcal^{G,\ell}(\Acal_{x_0^{(\ell)}})).
		\]
	\end{theorem}

	\section{Heegner points on Shimura curves}\label{Heegner points on Shimura curves}
	In this section, we define Heegner points on Shimura curves. We first fix some notation. We fix an imaginary quadratic number field $L$, and an ideal $\Nfrak$ of $\Z$ that is coprime to the discriminant $D_{L/\Q}$ of $L$. We write
	\[
	\epsilon_{L}\colon
	\A_\Q^\times/\Q^\times\to\{\pm1\}
	\]
	for the quadratic character attached to $L/\Q$. We assume the following weak Heegner hypothesis::
	\begin{taggedtheorem}{(wH)}\label{weak Heegner hypothesis}
		$\epsilon_{L}(\Nfrak)=1$.
	\end{taggedtheorem}
	Under Hypothesis~\ref{weak Heegner hypothesis}, there is a unique \emph{indefinite} quaternion algebra $B$ over $\Q$ that is ramified exactly at the primes $p$ of $\Q$ that are inert in $L$ with $\mathrm{ord}_{p}(\Nfrak)$ odd. We write $\Nfrak_B$ for the product of rational primes where $B$ is ramified. We fix an isomorphism
	\[
	B\otimes_\Q\R
	\simeq
	\mathrm{M}_2(\R).
	\]
	We view $G=B^\times$ as an algebraic group over $\Q$. We write $\Hfrak^\pm=\C\setminus\R$ for the union of the upper and lower half-planes. We let $G(\R)\simeq\GL_2(\R)$ act on $\Hfrak^\pm$ by fractional linear transformations. We then write $U_\infty$ for the stabilizer in $G(\R)$ of $i\in\Hfrak^\pm$. Thus we have an identification
	\[
	G(\R)/U_\infty\simeq\Hfrak^\pm.
	\]
	We have a section of the natural projection map $G(\R)\to\Hfrak^\pm$:
	\[
	s\colon
	\Hfrak^\pm\to  G(\R),
	\quad
	x+iy\mapsto
	\begin{pmatrix}
		y & x \\ 0 & 1
	\end{pmatrix}.
	\]
	
	By construction, for any place $v$ of $\Q$, if $B$ is ramified at $v$, then $L\otimes_\Q\Q_v$ is a field. Thus we have an embedding of $\Q$-algebras
	\[
	\kappa
	\colon
	L\hookrightarrow B.
	\]
	Moreover, there is a unique point $w(\kappa)\in\Hfrak^\pm$ that is fixed by $\kappa(x)$ for any $x\in L^\times$. We normalize $\kappa$ in such a way that for any $x\in L^\times$, we have
	\[
	\kappa(x)
	\begin{pmatrix}
		w(\kappa) \\ 1
	\end{pmatrix}
	=
	x
	\begin{pmatrix}
		w(\kappa) \\ 1
	\end{pmatrix},
	\]
	where on the left $\kappa(x)$ is viewed as a matrix in $\GL_2(\R)$ while on the right $x$ is viewed as a scalar in $\C^\times=(L\otimes_\Q\R)^\times$ (\cite[§1.2]{Howard2004b}).

	We fix a maximal order $\Ocal_B$ of $B$ containing $\kappa(\Ocal_L)$ and write $R=\kappa(\Ocal_L)+\kappa(\Nfrak)\Ocal_B$. Then $R$ has reduced discriminant equal to $\Nfrak$. Under the identification $B(\A_{\Q,f})^\times= G(\A_{\Q,f})$, we write $U$ for the image of $(R\otimes_\Z\widehat{\Z})^\times$ in $ G(\A_{\Q,f})$. The complex Shimura curve of level $U$ is given by
	\begin{align*}
		Sh_U(\C)
		&
		=
		G(\Q)\backslash  G(\A_\Q)/Z_ G(\A_\Q)U_\infty U
		\bigsqcup\{\text{cusps}\}
		\\
		&
		=
		G(\Q)\backslash
		\left(
		\Hfrak^\pm\times  G(\A_{\Q,f})/Z_ G(\A_{\Q,f})U_f
		\right)
		\bigsqcup\{\text{cusps}\}.
	\end{align*}
	It has a canonical model $ Sh_{U,\Q}$ over $\Q$, which is smooth, connected, and projective. The set of cusps is empty if and only if $B=\mathrm{M}_2(\Q)$. It is well known that $(G,\Hfrak^\pm)$ is a Shimura datum of Hodge type (\cite{Mumford1969}).

	For $(z,g)\in\Hfrak^\pm\times G(\A_{\Q,f})$, we write $[(z,g)]$ for its image in $Sh_{U}(\C)$. We let the normalizer of $U$ in $G(\A_{\Q,f})$ act on $Sh_{U}(\C)$ by the formula
	\[
	J(u)[(z,g)]:=[(z,gu^{-1})].
	\]
	Write $\Nrm\colon G(\A_\Q)\to\A_\Q^\times$ for the reduced norm. Then the set $\pi_0(Sh_{U}(\C))$ of connected components of $Sh_{U}(\C)$ can be identified with $\Q^\times\backslash\A_\Q^\times/\Nrm(Z_ G(\A_\Q)U_\infty U)$. Write $\Q_U/\Q$ for the abelian extension such that, under Artin reciprocity, there is an isomorphism
	\[
	\Gal(\Q_U/\Q)\simeq
	\Q^\times\backslash\A_\Q^\times/\Nrm(Z_ G(\A_\Q)U_\infty U).
	\]
	In what follows, we identify these two groups and define a map
	\[
	\GNrm\colon
	G(\A_{\Q,f})\to\Gal(\Q_U/\Q),
	\quad
	g\mapsto\Nrm(g)^{-1}.
	\]
	We let $\Gal(\Q_U/\Q)$ act on $\pi_0(Sh_{U}(\C))$ via $\GNrm$ (\cite[§1.2]{Howard2004b}):
	\[
	\begin{tikzcd}
		Sh_{U}(\C)
		\arrow[r,"J(u)"]
		\arrow[d,"\Nrm"']
		&
		Sh_{U}(\C)
		\arrow[d,"\Nrm"]
		\\
		\pi_0( Sh_{U}(\C))
		\arrow[r,"\GNrm(u)"]
		&
		\pi_0( Sh_{U}(\C))
	\end{tikzcd}
	\]

	We view $T=L^\times$ as an algebraic group over $\Q$, which is also a maximal torus of $G$ via the embedding $\kappa$. Then we put
	\[
	\CM_L=T(\Q)\backslash G(\A_{\Q,f})/Z_ G(\A_{\Q,f}),
	\]
	the set of CM points for $L$. We have a map
	\[
	\CM_L\to Sh_{U}(\C),
	\quad
	g\mapsto[(w(\kappa),g)].
	\]
	The image of this map is the set of CM points on $Sh_{U}(\C)$. We let $\Gal(L^\ab/L)$ act on the set of CM points by the formula
	\begin{equation}\label{Gal(L^ab/L) acts on CM_L}
		[(w(\kappa),g)]^{[s,L]}:=[(w(\kappa),s\cdot g)],
		\quad
		\forall
		s\in T(\A_{\Q,f}).
	\end{equation}
	Here $[-,L]\colon T(\Q)\backslash T(\A_{\Q,f})\simeq\Gal(L^\ab/L)$ is the Artin map. It is easy to see that $Z_ G(\A_{\Q,f})=\A_{\Q,f}^\times$ acts trivially on any CM point.

	For any CM point $x=[(w(\kappa),g)]$, the endomorphism ring of $x$ is the preimage of $R\otimes_\Z\widehat{\Z}$ under the map $L\to B\otimes_\Z\widehat{\Z}$ sending $h$ to $g^{-1}hg$. This endomorphism ring is an order of $L$ of the form $\Ocal_{\Cfrak_x}=\Z+\Cfrak_x\Ocal_L$ for some integral ideal $\Cfrak_x\subset\Z$ (the conductor of $x$). We write $T[\Cfrak_x]=\kappa(\widehat{\Ocal}_{\Cfrak_x}^\times)\subset T(\A_{\Q,f})$. We write $L[\Cfrak]$ for the ring class field of $L$, i.e., the abelian extension of $L$ corresponding to $T[\Cfrak]Z_ G(\A_{\Q,f})$ via class field theory. This is a Galois extension of $\Q$, and it is the field of definition of $x$.

	For each rational prime $\ell$ coprime to $\Nfrak D_{L/\Q}$, we fix an isomorphism
	\begin{equation}\label{R_ell and M_2(O_{F,ell})}
		R_\ell\simeq\mathrm{M}_2(\Z_\ell)
	\end{equation}
	such that $\kappa(\Ocal_{L,\ell})$ is identified with the set of diagonal matrices $\begin{pmatrix}
		x & 0 \\ 0 & y
	\end{pmatrix}$, respectively $\begin{pmatrix}
		x & yu_\ell \\ y & x
	\end{pmatrix}$ for $x,y\in\Z_\ell$, if $\ell$ splits in $L$, respectively if $\ell$ is inert in $L$. Here $u_\ell$ is a fixed non-square element in $\Z_\ell^\times$.
	This extends to an isomorphism $B_\ell\simeq\mathrm{M}_2(F_\ell)$. Fix a uniformizer $\varpi_\ell$ of $\Q_\ell$, and for each integer $k\ge0$, define an element $h[\ell^k]$ in $B_\ell$ whose image in $\mathrm{M}_2(\Q_\ell)$ is given by
	\[
	\begin{cases*}
		\begin{pmatrix}
			\varpi_\ell^k & 1 \\ 0 & 1
		\end{pmatrix},
		&
		if $\ell$ splits in $L$;
		\\
		\begin{pmatrix}
			\varpi_\ell^k & 0 \\ 0 & 1
		\end{pmatrix},
		&
		if $\ell$ is inert in $L$.
	\end{cases*}
	\]
	We view $h[\ell^k]$ as elements in $G(\A_{\Q,f})$ via the embedding $G(\Q_\ell)\hookrightarrow G(\A_{\Q,f})$. For any positive integer $n=\prod_{i=1}^r\ell_i^{k_i}$ coprime to $\Nfrak D_{L/\Q}$, we define
	\[
	h[n]:=\prod_{i=1}^rh[\ell_i^{k_i}].
	\]
	We view $h[n]$ also as elements in $\CM_L$ and write $x[n]$ for its image in $ Sh_{U}(\C)$.
	It is well known that these $x[n]$ are defined over the maximal abelian extension $L^\ab$ of $L$. Moreover, the action of $\Gal(L^\ab/L)$ on these $x[n]$ agrees with the action in (\ref{Gal(L^ab/L) acts on CM_L}) (\cite[Theorem 1.2.2]{Howard2004b}).

	We decompose $Sh_{U,\Q}\times_\Q\Q_U$ into geometric components $Sh_{U,\Q}\times_\Q\Q_U=\bigsqcup_{i} Sh_{U,\Q,i}$, with each $Sh_{U,\Q,i}$ geometrically irreducible. Moreover, for any extension $F/\Q$ such that $Sh_{U,\Q}(F)\neq\emptyset$, we must have $\Q_U\subset F$. We define
	\[
	J_{U,\Q}:=\Res^{\Q_U}_\Q(\mathrm{Jac}( Sh_{U,\Q,i_0}))
	\]
	for some fixed component $ Sh_{U,\Q,i_0}$. Here $\mathrm{Jac}(Sh_{U,\Q,i_0})$ is the Jacobian of $Sh_{U,\Q,i_0}$. Then $J_{U,\Q}$ is an abelian variety defined over $\Q$, which has good reduction away from $\Nfrak$. There is a unique element $\hfrak\in\mathrm{Pic}(Sh_{U,\Q})$ (up to a constant multiple) whose degree on each geometric component $Sh_{U,\Q,i}$ is constant and which satisfies $T_{a(m)}\hfrak=\degrm(T_{a(m)})\hfrak$ for every Hecke operator $T_{a(m)}$ with a positive integer $m$ coprime to $\Nfrak D_{L/\Q}$. This is called the Hodge class, and on each geometric component $\hfrak$ is just the canonical divisor. We write $\hfrak_i$ for the restriction of $\hfrak$ to the component $Sh_{U,\Q,i}$; then we have a unique morphism defined over $\Q$
	\[
	Sh_{U,\Q}\to J_{U,\Q},
	\]
	which, on complex points, takes $x_i\in  Sh_{U,\Q}(\C)$ to the divisor $dx_i-\hfrak_i\in J_{U,\Q}(\C)$.

	Let $A/\Q$ be an abelian variety with a surjective morphism of abelian varieties $\alpha\colon J_{U,\Q}\to A$. Composed with the map $ Sh_{U,\Q}\to J_{U,\Q}$, we have a morphism over $\Q$
	\[
	\pi\colon Sh_{U,\Q}\to J_{U,\Q}\to A.
	\]
	Since $ Sh_{U,\Q}$ is complete and connected, the morphism $\pi$ is either finite or constant, the latter of which is impossible since $\alpha$ is surjective. Thus $\pi$ is a finite morphism (\emph{cf.} \cite[Lemma 3.9]{CornutVatsal2007}). We write
	\[
	y[n]=\pi(x[n])
	\in
	A(L[n]).
	\]

	As in the introduction, we write $L[\infty]=\bigcup_kL[k]$.

	\section{Relation between Hecke orbits and Galois orbits}\label{Hecke orbit and Galois orbit}
	In this section, we study the relationship between Hecke orbits and Galois orbits.
	
	We keep the notation as in §\ref{Heegner points on Shimura curves}. We fix a prime $\ell$ ($\Scal=\{\ell\}$ in the notation of §\ref{Hecke operators}) that is inert in $L$, and a place $v(\ell)$ of $\overline{\Q}$ over $\ell$. We fix the point $x_0=x[1]\in Sh_{U,\Q}(\overline{\Q})$ and assume that it has \emph{supersingular} good reduction at $v(\ell)$. We fix a non-empty finite subset $\Tcal=\{\sigma_1,\cdots,\sigma_r\}$ of
	\[
	Z_{I_{x_0}}(\Q)\backslash Z_{I_{x_0}}(\A_f)=I_{x_0}(\Q)\backslash I_{x_0}(\A_f)=L^\times\backslash \widehat{L}^\times.
	\]
	We identify $\Tcal$ with its image in $\Gal(L[\infty]/L)$ via the Artin reciprocity map:
	\begin{equation}\label{Artin map}
		\Art
		\colon
		\Gal(L[\infty]/L)
		\simeq
		\widehat{L}^\times/L^\times.
	\end{equation}
	We have a partition $\Tcal=\bigsqcup_{i=1}^k\Tcal_i$ as in §\ref{Shimura varieties of Hodge type}. We set $\Gbf_0=I_{x_0^{(\ell)}}$, an algebraic group over $\Q$.
	We fix isomorphisms
	\begin{equation}\label{G_0(Z_q)=R_q}
		\Gbf_0(\Z_q)
		\simeq
		G(\Z_q),
		\quad
		q\neq\ell.
	\end{equation}

	Consider a positive integer $n$ coprime to $\ell\Nfrak D_{L/\Q}$. We first consider Galois orbits. There are the following cases:
	\begin{enumerate}		
		\item 
		$n=q^e$ is a prime power with $q$ a prime number that splits in $L$. We fix a set $S_n$ of representatives in $\Z_q^\times$ of the quotient set $(\Z_q/n\Z_q)^\times$. Note that $S_n$ can be viewed as a set of representatives of the quotient $(\Ocal_{L,q}/n\Ocal_{L,q})^\times/(\Z_q/n\Z_q)^\times$ because $\Ocal_{L,q}\simeq\Z_q^2$ as $\Z_q$-algebras. Moreover, $\#S_n=q^{e-1}(q-1)$. Under the identification (\ref{R_ell and M_2(O_{F,ell})}), we define a subset $\Gfrak_n$ of $\Gal(L[\infty]/L)$ such that under the Artin map (\ref{Artin map}), we have
		\[
		\Art(\Gfrak_n)
		=
		L^\times\cdot
		\left\{
		\begin{pmatrix}
			a & 0 \\ 0 & 1
		\end{pmatrix}|a\in S_n
		\right\}.
		\]

		\item 
		$n=q^e$ is a prime power with $q$ a prime number that is inert in $L$. We fix a set $S_n$ of representatives in $\Ocal_{L,q}^\times$ of the quotient $(\Ocal_{L,q}/n\Ocal_{L,q})^\times/(\Z_q/n\Z_q)^\times$. Moreover, $\#S_n=q^{e-1}(q+1)$. We define a subset $\Gfrak_n$ of $\Gal(L[\infty]/L)$ such that
		\[
		\Art(\Gfrak_n)
		=
		L^\times\cdot S_n.
		\]
		Explicitly, we can take $S_n$ to be (under the identification (\ref{R_ell and M_2(O_{F,ell})}))::
		\[
		S_n
		=
		\left\{
		\begin{pmatrix}
			1 & bu_q \\ b & 1
		\end{pmatrix}|b=1,\cdots,q^e
		\right\}
		\bigsqcup
		\left\{
		\begin{pmatrix}
			pau_q & u_q \\ 1 & pau_q
		\end{pmatrix}|a=1,\cdots,q^{e-1}
		\right\}.
		\]

		\item 
		In general, for a prime factorization $n=\prod_{i=1}^sp_i^{e_i}$, we define $S_n=\prod_{i=1}^sS_{p_i^{e_i}}$, viewed as a subset of $\widehat{\Ocal}_L^\times$. We then define a subset $\Gfrak_n$ of $\Gal(L[\infty]/L)$ such that
		\[
		\Art(\Gfrak_n)
		=
		L^\times\cdot S_n.
		\]
	\end{enumerate}

	We next consider Hecke orbits. As above, there are the following cases::
	\begin{enumerate}
		\item 
		$n=q^e$ is a prime power. We define
		\[
		a(n)
		:=
		\begin{pmatrix}
			n & 0 \\ 0 & 1
		\end{pmatrix}
		\in
		\GL_2(\Q_q)
		\simeq
		G(\Q_q)
		\hookrightarrow
		G(\A_f^{(\ell)}).
		\]
		By definition, the double coset decomposes into right cosets
		\[
		G(\Z_q)a(n)G(\Z_q)
		=
		\bigsqcup_{g\in S_n'}a(n)G(\Z_q),
		\]
		where $S_n'$ is a set of representatives in $G(\Z_q)$ of the quotient $G(\Z_q)/(G(\Z_q)\bigcap a(n)G(\Z_q)a(n)^{-1})$. Note that $\#S_n'=q^{e-1}(q+1)=\deg(a(n))$ (\cite[Theorem 3.24]{Shimura1971}). We have
		\begin{lemma}\label{Galois orbit contained in Hecke orbit}
			For any distinct $g,g'\in S_n$, the right cosets
			$gh[n]G(\Z_q)$ and $g'h[n]G(\Z_q)$ are distinct right cosets in the decomposition $G(\Z_q)a(n)G(\Z_q)$ (here we use the identifications (\ref{G_0(Z_q)=R_q})).
		\end{lemma}
		\begin{proof}
			We first show that the two cosets are distinct.
			If $q$ splits in $L$, write $g=\begin{pmatrix}
				a & 0 \\ 0 & 1
			\end{pmatrix}$ and $g'=\begin{pmatrix}
				a' & 0 \\ 0 & 1
			\end{pmatrix}$. Then we have
			\[
			gh[n]G(\Z_q)
			=
			\begin{pmatrix}
				a\varpi_q^e & a \\ 0 & 1
			\end{pmatrix}
			G(\Z_q)
			=
			\begin{pmatrix}
				n & a \\ 0 & 1
			\end{pmatrix}
			G(\Z_q),
			\quad
			g'h[n]G(\Z_q)
			=
			\begin{pmatrix}
				n & a' \\ 0 & 1
			\end{pmatrix}
			G(\Z_q).
			\]
			Since 
			\[
			\begin{pmatrix}
				n & a \\ 0 & 1
			\end{pmatrix}^{-1}
			\begin{pmatrix}
				n & a' \\ 0 & 1
			\end{pmatrix}
			=
			\begin{pmatrix}
				1 & (a'-a)/n \\ 0 & 1
			\end{pmatrix}\notin G(\Z_q),
			\]
			these right cosets are indeed distinct.

			If $q$ is inert in $L$, write
			$g^{-1}g'=\begin{pmatrix}
				a & b \\ c & d
			\end{pmatrix}$. By definition of $S_n$, one checks that $b/n\notin\Z_q$. Then we have
			\[
			(gh[n])^{-1}(g'h[n])
			=
			\begin{pmatrix}
				a & b/n \\ cn & d
			\end{pmatrix}\notin G(\Z_q).
			\]
			It follows again that these two right cosets are distinct.

			Next we show that both cosets lie in $G(\Z_q)a(n)G(\Z_q)$. By the Cartan decomposition for $G(\Q_q)$, we can characterize $G(\Z_q)a(n)G(\Z_q)$ as the subset of $G(\Q_q)$, consisting of matrices $A=\begin{pmatrix}
				a & b \\ c & d
			\end{pmatrix}$ with $a,b,c,d\in\Z_q$ such that $\det(A)\in n\Z_q^\times$ and $a,b,c,d$ are not simultaneous divisible by $q$. Using this criterion, it follows easily that $gh[n]G(\Z_q)$ and $g'h[n]G(\Z_q)$ both lie in $G(\Z_q)a(n)G(\Z_q)$.
		\end{proof}
		In particular, if $q$ is inert in $L$, then we can take $S_n'$ to be $S_n$; if $q$ splits in $L$, we can take $S_n'$ to contain $S_n$.

		\item 
		In general, for a prime factorization $n=\prod_{i=1}^sp_i^{e_i}$, we have (\cite[Theorem 3.24]{Shimura1971})
		\[
		T_{a(n)}
		=
		\prod_{i=1}^sT_{a(p_i^{e_i})}.
		\]
		It follows from the previous case that for any distinct $g,g'\in S_n$, the right cosets $gh[n]U$ and $g'h[n]U$ are distinct, and both lie in $Uh[n]U$ (note that for any $p_i$, $U_{p_i}=G(\Z_{p_i})$ under the identification (\ref{G_0(Z_q)=R_q}) because $n$ is coprime to $\ell\Nfrak D_{L/\Q}$).
	\end{enumerate}

	For a positive integer $n=\prod_{i=1}^sp_i^{e_i}$ coprime to $\ell\Nfrak D_{L/\Q}$, we define a rational number
	\[
	d(n):=\frac{\#S_n}{\#S_n'}=\frac{\#S_n}{\deg(a(n))}\in(0,1].
	\]
	In particular, $d(n)=1$ if and only if all the prime factors $p_i$ are inert in $L$.

	In summary, we have obtained::
	\begin{theorem}\label{theorem on Hecke orbit vs Galois orbit}
		For any $n$ coprime to $\ell\Nfrak D_{L/\Q}$, the Hecke orbit $T_n(x[1]^{(\ell)})$ contains the Galois orbit $\Gfrak_n(x[n]^{(\ell)})$. Moreover, we have $d(n)=\frac{\#\Gfrak_n}{\deg(a(n))}$.
	\end{theorem}

	\section{Variants of Mazur's conjecture-1}\label{Mazur's conjecture-1}
	Recall that we have fixed a moduli parametrization $\pi\colon Sh_{U,\Q}\to A$. This is a finite morphism (\emph{cf.} \cite[Theorem 3.9]{CornutVatsal2007} or \cite[§9]{Zhang2025}). By \cite[Lemma 4.1]{Cornut2002}, we know that $A(E[\infty])_\mathrm{tors}$ is finite, say, of cardinality $t_1$. We write $t_2$ for an upper bound for the cardinalities of the fibers of the map $\pi\colon Sh_{U}\to A$, which is finite since $\pi$ is a finite morphism.

	We write $Sh_{U}^\ssrm(\F_{v(\ell)})$ for the supersingular locus of $Sh_{U}(\F_{v(\ell)})$, where $\F_{v(\ell)}=\overline{\F}_\ell$ is the residue field of $\Ocal_{\overline{\Q}}$ at the place $v(\ell)$. Then we have bijections (\cite[Corollary 3.13 \& §3.2.5]{CornutVatsal2005})
	\[
	Sh_{U}^\ssrm(\overline{\F}_\ell)
	\simeq
	I_{x_0}(\Q)\backslash G(\A_f^{\ell})U/U
	\simeq
	\Hcal^{G,\ell}(\Acal_{x[1]^{(\ell)}}).
	\]
	By Eichler's mass formula, we know that as $\ell\to +\infty$, $\#Sh_{U}^\ssrm(\overline{\F}_\ell)\to +\infty$. We choose $\ell$ \emph{inert} in $L$ such that
	\begin{equation}\label{cardinal of Sh_U^ss}
		\#Sh_U^\ssrm(\overline{\F}_\ell)>(t_1\times r!)^{2\dim(A)}\times t_2.
	\end{equation}

	In what follows, we fix $x_0=x[1]$ and assume that $x[1]$ has \emph{supersingular} reduction $x[1]^{(\ell)}$ at $v(\ell)$. For each element $g\in I_{x[1]}(\Q)\backslash G(\A_f^{(\ell)})U$, we write $c_g=c_g(\ell)$ for the cardinality of the stabilizer of $g$ in $U$ (where $U$ acts on the right on $I_{x[1]^{(\ell)}}(\Q)\backslash G(\A_f^{(\ell)})U$). This is finite because $I_{x[1]^{(\ell)}}(\R)$ is compact modulo its center. Choose a set of representatives $\{g_1,\cdots,g_s\}$ of $I_{x[1]^{(\ell)}}(\Q)\backslash G(\A_f^{(\ell)})U/U$ and set
	\begin{equation}\label{c(pi,ell,T)}
		c(\pi,\ell,\Tcal):=\min_i
		\left(\frac{c_{g_i}^{-1}}{\sum_{j=1}^sc_{g_j}^{-1}}\right)^r
		\in(0,1).
	\end{equation}
	Here $k$ is the cardinal of the partition $\Tcal=\bigsqcup_{i=1}^k\Tcal_i$ as in §\ref{Shimura varieties of Hodge type}.

	We have the following result.
	\begin{theorem}\label{Galois orbit equidistributes}
		Let $\ell$ and $\pi\colon Sh_U\to A$ be as above. For any integer $n\gg1$ coprime to $\ell\Nfrak D_{L/\Q}$ with $1-d(n)<c(\pi,\ell,\Tcal)$, we have
		\[
		\Rbf_{\Tcal}(\{\sigma(x[n])|\sigma\in \Gfrak_n\})
		=
		\prod_{i=1}^k\widetilde{\Delta}^{\Tcal_i}(Sh_U^\ssrm(\overline{\F}_\ell)).
		\]

	\end{theorem}
	\begin{proof}
		We enumerate the integers $n$ coprime to $\ell\Nfrak D_{L/\Q}$ as $a(1)<a(2)<\cdots$. We want to apply Theorem \ref{T_{a_n(x_0)} equidistribute on Shimura variety} to the sequence $a(1),a(2),\cdots$. For this, we verify that the conditions in that theorem are all satisfied:
		\ref{T_{a_n(x_0)} equidistribute on Shimura variety}(2) and \ref{T_{a_n(x_0)} equidistribute on Shimura variety}(3) are clearly satisfied; for \ref{T_{a_n(x_0)} equidistribute on Shimura variety}(1), in (\ref{condition on similitude of I_x}), the left hand side is equal to $(\A_f^{(\ell)})^\times$; on the right hand side, $\pr(I_{x[1]^{(\ell)}}(\widehat{\Z}^{(\ell)}))
		=(\widehat{\Z}^{(\ell)})^\times$ while
		$\pr(I_{x[1]^{(\ell)}}(\Q)\bigcap U_\ell)
		=\Z_{(\ell)}^\times$ by \cite[p.90]{Vigneras1980}, thus (\ref{condition on similitude of I_x}) holds; for \ref{T_{a_n(x_0)} equidistribute on Shimura variety}(4), we apply \cite[Proposition 3.4]{Cornut2002} and we deduce that (\ref{G vs G^1 for X}) holds. As a result, we can indeed apply Theorem \ref{T_{a_n(x_0)} equidistribute on Shimura variety} and thus as $n\to+\infty$, the sequence $(\Rbf_{\Tcal}(T_{a(n)}(x[1]^{(\ell)})))_n$ equidistributes on the following \emph{finite} subset of $\prod_{i= 1}^r\Hcal^{G,\ell}(\Acal_{x[1]^{(\ell)}})$
		\[
		\prod_{i=1}^k\widetilde{\Delta}^{\Tcal_i}(\Hcal^{G,\ell}(\Acal_{x[1]^{(\ell)}}))
		\simeq
		\prod_{i=1}^k\widetilde{\Delta}^{\Tcal_i}(Sh_U^\ssrm(\overline{\F}_\ell)).
		\]
		In particular, for any $0<\epsilon\ll1$, there is $n_0$ such that for any $n\ge n_0$ and any $\underline{h}=(h_1,\cdots,h_r)\in\prod_{i=1}^k\widetilde{\Delta}^{\Tcal_i}(Sh_U^\ssrm(\overline{\F}_\ell))$, the number $f(n,\underline{h})$ of points $x'$ in $T_{a(n)}(x[1]^{(\ell)})$ (counted with multiplicities) with $\Rbf_{\Tcal}(x')=\underline{h}$ satisfies
		\begin{equation}\label{|f(n,h)/S_n'-...|<epsilon}
			\left|
			\frac{f(n,\underline{h})}{\#S_n'}-\prod_{i=1}^r\frac{c_{h_i}^{-1}}{\sum_{j=1}^sc_{g_j}^{-1}}
			\right|
			=
			\left|
			\frac{f(n,\underline{h})}{\deg(a(n))}-\prod_{i=1}^r\frac{c_{h_i}^{-1}}{\sum_{j=1}^sc_{g_j}^{-1}}
			\right|<\epsilon.
		\end{equation}

		Note that the set of values of $d(n)$ is discrete away from $1$, so there exists $0<\epsilon\ll1$ such that for any $n'$, if $1-d(n')<c(\pi,\ell,\Tcal)$, then $1-d(n')<c(\pi,\ell,\Tcal)-\epsilon$. We choose such an $\epsilon$; then by the assumption on $n$, we have (\emph{cf.} Lemma \ref{Galois orbit contained in Hecke orbit})
		\[
		\frac{\#(S_n'\backslash S_n)}{\#S_n'}
		=
		1-d(n)<c(\pi,\ell,\Tcal)-\epsilon.
		\]
		So combining with \ref{f(n)/(S_n')^2}, we know that for any $\underline{h}\in\prod_{i=1}^k\widetilde{\Delta}^{\Tcal_i}(Sh_U^\ssrm(\overline{\F}_\ell))$, there is at least one point $x'$ in the subset $\Gfrak_n(x[n]^{(\ell)})=\{\sigma(x[n])^{(\ell)}|\sigma\in\Gfrak_n\}$ of $T_{a(n)}(x[1]^{(\ell)})$ with $\Rbf_{\Tcal}(x')=\underline{h}$. This finishes the proof of the theorem.
	\end{proof}

	We deduce immediately the following.
	\begin{corollary}\label{generalized Mazur's conjecture}
		Let $\ell$ and $\pi\colon Sh_U\to A$ be as above. For any integer $n\gg1$ coprime to $\ell\Nfrak D_{L/\Q}$ with $1-d(n)<c(\pi,\ell,\Tcal)$, we have
		\[
		\sum_{\sigma\in\Tcal}\sigma(y[n])
		\notin
		A(L[\infty])_\mathrm{tors}.
		\]
	\end{corollary}
	\begin{proof}
		The proof is very similar to \cite[Theorem 9.2]{Zhang2025}. We argue by contradiction, and thus assume that $\sum_{\sigma\in\Tcal}\sigma(y[n])$ is torsion for every $n$ coprime to $\ell\Nfrak D_{L/\Q}$ with $1-d(n)<c(\pi,\ell,\Tcal)$.

		Let $n$ be as in the corollary. By Theorem \ref{Galois orbit equidistributes}, for any $x,x'\in Sh_U^\ssrm(\overline{\F}_\ell)$, we can choose $\nu,\nu'\in\Gfrak_n$ such that the components of $\Rbf_{\Tcal}(\sigma(\nu(x[n])))$ are equal to those of $\Rbf_{\Tcal}(\sigma(\nu'(x[n])))$, except at the components in $\Tcal_{i_0}\subset\Tcal$ for some $i_0$, in which case
		\[
		\Rbf_{\Tcal}(\sigma(\nu(x[n])))=x
		\quad
		\text{and}
		\quad
		\Rbf_{\Tcal}(\sigma(\nu'(x[n])))=x',
		\quad
		\forall
		\sigma\in\Tcal_{i_0}.
		\]
		In particular, we have
		\[
		\Rbf_{\Tcal}(\sigma(\nu(y[n])))
		-
		\Rbf_{\Tcal}(\sigma(\nu'(y[n])))
		=
		\left(
		0,\cdots,0,
		\pi(x)-\pi(x'),\cdots,\pi(x)-\pi(x'),
		0,\cdots,0
		\right).
		\]
		Here on the right hand side, all the components are zero except at the components in $\Tcal_{i_0}$, which are $\pi(x)-\pi(x')$. Then we have
		\[
		\sum_{\sigma\in\Tcal}\sigma(\nu(y[n]))^{(\ell)}
		-
		\sum_{\sigma\in\Tcal}\sigma(\nu'(y[n]))^{(\ell)}
		=
		\sum_{\sigma\in\Tcal_{i_0}}(\pi(x)-\pi(x'))
		=
		\#\Tcal_{i_0}(\pi(x)-\pi(x')).
		\]
		Note that for any prime $q|t_1$,
		\[
		\dim_{\F_q}(A(\overline{\F}_\ell)\otimes\F_q)=2\dim(A).
		\]
		Thus $\#A(\overline{\F}_\ell)[t_1]\leq t_1^{2\dim(A)}$. By assumption on $\ell$ (\emph{cf.} (\ref{cardinal of Sh_U^ss})), $\#Sh_U^\ssrm(\overline{\F}_\ell)>(t_1\times r!)^{2\dim(A)}\times t_2$. So there exists $x,x'\in Sh_U^\ssrm(\overline{\F}_\ell)$ such that
		$\pi(x)-\pi(x')\in A(\overline{\F}_\ell)$ has order not dividing $t_1\times r!$.

		On the other hand, we have assumed that $\sum_{\sigma\in\Tcal}\sigma(y[n])$ is torsion; hence both $\sum_{\sigma\in\Tcal}\sigma(\nu(y[n]))$ and $\sum_{\sigma\in\Tcal}\sigma(\nu'(y[n]))$ are torsion, and so are their reductions modulo $v(\ell)$, namely $\sum_{\sigma\in\Tcal}\sigma(\nu(y[n]))^{(\ell)}$ and $\sum_{\sigma\in\Tcal}\sigma(\nu'(y[n]))^{(\ell)}$. In particular, the following element in $A(\overline{\F}_\ell)$,
		\[
		\sum_{\sigma\in\Tcal}\sigma(\nu(y[n]))^{(\ell)}-\sum_{\sigma\in\Tcal}\sigma(\nu'(y[n]))^{(\ell)}
		=
		\#\Tcal_{i_0}(\pi(x)-\pi(x'))
		\]
		is a torsion element, whose order necessarily divides $t_1=\#A(L[\infty])_\mathrm{tors}$. So $\pi(x)-\pi(x')$ is a torsion element, whose order necessarily divides $t_1\times r!$,\footnote{More precisely, we can deduce that the order of $\pi(x)-\pi(x')$ divides $t_1\times\#\Tcal_{i_0}$. But for our argument in the proof, the multiple $t_1\times r!$ of $t_1\times\#\Tcal_{i_0}$ suffices.} which is a contradiction to the preceding paragraph. This finishes the proof of the corollary.
	\end{proof}

	We then deduce some consequences from the above.

	We fix a prime $p$ and write $L[p^\infty]=\bigcup_kL[p^k]$. This is a finite extension over the maximal anticyclotomic $\Z_p$-extension $H[p^\infty]$ of $L$. We take $\Tcal$ to be a set of representatives in $\widehat{L}^\times$ of the subgroup $\Gal(L[p^\infty]/H[p^\infty])$ of $\Gal(L[p^\infty]/L)$ (under the Artin map (\ref{Artin map})). Then we have
	\begin{corollary}\label{vertical Mazur's conjecture}
		Let $\ell$ and $\pi\colon Sh_U\to A$ be as above. Suppose that $p$ is either inert in $L$ or splits in $L$ with $\frac{2}{p+1}<c(\pi,\ell,\Tcal)$. Then for $k\gg1$, we have
		\[
		\sum_{\sigma\in\Tcal}\sigma(y[p^k])\notin
		A(L[\infty])_\mathrm{tors}.
		\]
	\end{corollary}
	\begin{proof}
		This is a direct application of Corollary \ref{generalized Mazur's conjecture} by noting that $d(p^k)=d(p)$ and it is equal to $1$ if $p$ is inert in $L$ and equal to $(p-1)/(p+1)$ if $p$ splits in $L$.
	\end{proof}
	This is a strengthening of Mazur's original conjecture (\cite{Mazur1983}).

	Now we take $\Tcal$ to be a set of representatives in $\widehat{L}^\times$ for the quotient $\widehat{L}^\times/L^\times\widehat{\Ocal}_L^\times$.	
	\begin{corollary}\label{horizontal Mazur's conjecture}
		Let $\pi\colon Sh_U\to A$ be as above. For any prime $p\gg1$, we have
		\[
		\sum_{\sigma\in\Tcal}\sigma(y[p])\notin A(L[\infty])_\mathrm{tors}.
		\]
	\end{corollary}
	\begin{proof}
		Indeed, we can assume $p>\ell\Nfrak D_{L/\Q}$. Moreover, $1-d(p)<\frac{2}{p}$. For $p\gg1$, $1-d(p)\ll1$. In particular, we can take $p$ such that $1-d(p)<c(\pi,\ell,\Tcal)$. Then we can apply Corollary \ref{generalized Mazur's conjecture}.
	\end{proof}
	This can be seen as a horizontal version of Mazur's conjecture.

	\begin{remark}
		All the results in §\ref{Heegner points on Shimura curves} and this section can be generalized to Shimura curves over totally real number fields with appropriate modifications in the arguments. We restrict ourselves to the case of Shimura curves over $\Q$ for ease of notation.
	\end{remark}

	\section{Variants of Mazur's conjecture-2}\label{Mazur's conjecture-2}
	In the previous section, we fixed an auxiliary prime $\ell$ and considered positive integers $n$ coprime to $\ell$ (as well as coprime to $\Nfrak D_{L/\Q}$). In this section, we fix an infinite set of such primes $\ell$ and prove a stronger result for positive integers $n$ coprime to all primes in this set.

	Let $\Tcal$ be as before.
	We fix an \emph{infinite} set $\Scal^\circ$ of primes $\ell$ and, for each $\ell\in\Scal^\circ$, a place $v(\ell)$ of $\overline{\Q}$ over $\ell$. We assume that $x_0=x[1]$ has \emph{supersingular} reduction at $v(\ell)$ for all $\ell\in\Scal^\circ$.\footnote{Such an $\Scal^\circ$ always exists because $x_0$ is a CM point in the Shimura curve $Sh_U$.} Then, for $\ell\in\Scal^\circ$, as $\ell\to +\infty$, we have $\#Sh_U^\ssrm(\overline{\F}_\ell)\to +\infty$. Moreover, it is well known that the cardinalities $c_g(\ell)$ are bounded for all $g$ and all $\ell\in\Scal^\circ$ (in fact, $c_g(\ell)\le 24$; \emph{cf.} \cite[Theorem 11.5.14]{Voight2021}).

	Then we have the following:
	\begin{theorem}\label{S^circ}
		Let $\Scal^\circ$ and $\pi\colon Sh_U\to A$ be as above. For any $0<c\ll1$, there is an integer $n_0>0$ such that for any $n>n_0$ coprime to all primes in $\Scal^\circ$ and $\Nfrak D_{L/\Q}$ with $1-d(n)<c$, we have
		\[
		\sum_{\sigma\in\Tcal}\sigma(y[n])\notin A(L[\infty])_\mathrm{tors}.
		\]
	\end{theorem}
	\begin{proof}
		The idea of the proof is similar to that of Theorem \ref{Galois orbit equidistributes}. We choose $\ell\in\Scal^\circ$ sufficiently large, so $s=s(\ell)=\#Sh_U^\ssrm(\overline{\F}_\ell)\gg1$ and thus $c(\pi,\ell,\Tcal)\ll1$ because each $c_g(\ell)$ is bounded independently of $g$ and $\ell$. Then for any $n\gg1$ coprime to all primes in $\Scal^\circ$ and $\Nfrak D_{L/\Q}$, $\Rbf_{\Tcal}(T_{a(n)}(x[1]^{(\ell)}))$ is equal to $\prod_{i=1}^k\widetilde{\Delta}^{\Tcal_{i}}(Sh_U^\ssrm(\overline{\F}_\ell))$ and moreover for such $n\to+\infty$, the sets $(\Rbf_{\Tcal}(\{\sigma(x[n])|\sigma\in\Gfrak_n\}))_n$ \emph{equidistributes} on $\prod_{i=1}^k\widetilde{\Delta}^{\Tcal_{i}}(Sh_U^\ssrm(\overline{\F}_\ell))$.

		For ease of notations, we write $Sh(\ell)=\prod_{i=1}^k\widetilde{\Delta}^{\Tcal_{i}}(Sh_U^\ssrm(\overline{\F}_\ell))$. We then write $X(\ell)$ for the subset of $Sh(\ell)^2$, consisting of pairs of points $(\underline{h},\underline{h}')$ in $Sh(\ell)$ such that $\sum_{\sigma\in\Tcal}(\pi(h_\sigma)-\pi(h'_\sigma))$ lies in the image of $A(L[\infty])_\mathrm{tors}$ in $A(\overline{\F}_\ell)$. For any $\underline{h}\in Sh(\ell)$, there are at most $(\#Sh_U^\ssrm(\overline{\F}_\ell))^{k-1}\times\#A(L[\infty])_\mathrm{tors}$ many points $\underline{h}'\in Sh(\ell)$ such that $(\underline{h},\underline{h}')\in X(\ell)$. Thus we know that
		\[
		\#X(\ell)\leq s(\ell)^k\times(s(\ell)^{k-1}\times\#A(L[\infty])_\mathrm{tors}).
		\]
		It follows that
		\[
		\frac{\#X(\ell)}{\#Sh(\ell)^2}
		\leq
		\frac{\#A(L[\infty])_\mathrm{tors}}{s(\ell)}\ll1.
		\]
		Thus for any $n\gg1$ coprime to all primes in $\Scal^\circ$ and $\Nfrak D_{L/\Q}$, the number $f(n)$ of pairs of elements $(g(x[1]),g'(x[1]))\in T_{a(n)}(x[1])$ such that $(\Rbf_{\Tcal}(\pi(g(x[1]))),\Rbf_{\Tcal}(\pi(g'(x[1]))))\in X(\ell)$ satisfies
		\[
		\frac{f(n)}{\#(S_n')^2}\ll1.
		\]

		On the other hand, since $d(n)>1-c$, we have
		\begin{equation}\label{S_n^2/(S_n')^2}
			\frac{\#(S_n)^2}{\#(S_n')^2}=d(n)^2>(1-c)^2.
		\end{equation}
		We choose $n_0=n_0(\ell)$ such that for any $n>n_0$ coprime to all primes in $\Scal^\circ$ and $\Nfrak D_{L/\Q}$, we have
		\begin{equation}\label{f(n)/(S_n')^2}
			\frac{f(n)}{\#(S_n')^2}<(1-c)^2.
		\end{equation}
		Combining (\ref{S_n^2/(S_n')^2}) and (\ref{f(n)/(S_n')^2}), it follows that there exists $(\underline{h},\underline{h}')\in\Rbf_{\Tcal}(\Gfrak_n(x[n]))^2$ that is not in $X(\ell)$. In other words, we can find $\nu,\nu'\in\Gfrak_n$ such that
		\[
		\sum_{\sigma\in\Tcal}(\sigma(\nu(y[n]))-\sigma(\nu'(y[n])))\notin
		A(L[\infty])_\mathrm{tors}.
		\]
		In particular, either $\sum_{\sigma\in\Tcal}\sigma(\nu(y[n]))$ or $\sum_{\sigma\in\Tcal}\sigma(\nu'(y[n]))$ is non-torsion in $A(L[\infty])$. In either case, $\sum_{\sigma\in\Tcal}\sigma(y[n])$ is non-torsion, which finishes the proof of the theorem.
	\end{proof}

	An immediate corollary is the following strengthening of Corollary \ref{vertical Mazur's conjecture}.
	\begin{corollary}\label{vertical Mazur's conjecture, stronger version}
		Let $\pi\colon Sh_U\to A$ be as above. Fix a prime $p$ coprime to $\Nfrak D_{L/\Q}$. Then for $k\gg1$, we have
		\[
		\sum_{\sigma\in\Tcal}\sigma(y[p^k])
		\notin
		A(L[\infty])_\mathrm{tors}.
		\]
	\end{corollary}
	\begin{proof}
		For $\Scal^\circ$ as above, we apply Theorem \ref{S^circ} to $\Scal^\circ\backslash\{p\}$ instead of $\Scal^\circ$.
	\end{proof}

\end{document}